\documentclass[11pt,oneside]{amsart}

\usepackage[margin=1in]{geometry}
\usepackage{amsmath,amssymb,amsthm}
\usepackage{tikz}
\usepackage{float}
\usetikzlibrary{arrows.meta,calc,positioning}
\usepackage[colorlinks=true,citecolor=blue,linkcolor=blue,urlcolor=blue,pagebackref]{hyperref}

\numberwithin{equation}{section}

\theoremstyle{plain}
\newtheorem{theorem}{Theorem}[section]
\newtheorem{corollary}[theorem]{Corollary}
\newtheorem{proposition}[theorem]{Proposition}
\newtheorem{lemma}[theorem]{Lemma}

\theoremstyle{definition}
\newtheorem{definition}[theorem]{Definition}

\theoremstyle{remark}
\newtheorem{remark}[theorem]{Remark}
\newtheorem{question}[theorem]{Question}

\title{Minkowski sums with convex curves without pointwise Fourier decay}

\author{A. Iosevich}
\address{Department of Mathematics, University of Rochester, Rochester, New York 14627, USA}
\email{alex.iosevich@rochester.edu}

\author{Z. Li}
\address{Department of Mathematics, The Ohio State University, Columbus, Ohio 43210, USA}
\email{li.15128@osu.edu}

\author{E. Palsson}
\address{Department of Mathematics, Virginia Tech, Blacksburg, Virginia 24061, USA}
\email{palsson@vt.edu}

\author{K. Taylor}
\address{Department of Mathematics, The Ohio State University, Columbus, Ohio 43210, USA}
\email{taylor.2952@osu.edu}

\author{A. Yavicoli}
\address{Department of Mathematics, University of British Columbia, Vancouver, British Columbia V6T 1Z2, Canada}
\email{yavicoli@math.ubc.ca}



\hypersetup{
  pdftitle={Minkowski sums with convex curves without pointwise Fourier decay},
  pdfauthor={A. Iosevich, Z. Li, E. Palsson, K. Taylor, and A. Yavicoli},
  pdfsubject={Positive measure of Minkowski sums with convex curves without pointwise Fourier decay},
  pdfkeywords={Minkowski sums, convex curves, curve projections, Favard curve length, curvature measures, tube overlaps, Fourier decay, rectifiability}
}

\begin{document}

\begin{abstract}
Let $\Gamma\subset\mathbb R^2$ be a compact convex graph and define
\[
T(\Gamma)
=
\inf
\left\{
 t:
 \dim_{\mathrm H}(E)>t
 \Longrightarrow
 |E+\Gamma|>0
 \text{ for every compact }E\subset\mathbb R^2
\right\}.
\]
For a graph over an interval of positive length the smallest possible value is $T(\Gamma)=1$.
We ask whether this optimal conclusion can hold when pointwise Fourier decay
of arclength is unavailable. The answer is yes, even for strictly convex
curves. We use the Fourier transform convention
$\widehat\nu(\xi)=\int e^{-2\pi i x\cdot\xi}\,d\nu(x)$. We construct a strictly
convex Lipschitz graph $\Gamma$ with
$T(\Gamma)=1$ such that, for every nontrivial subarc $\Gamma_0$ and every
$\alpha>0$,
\[
\limsup_{|\xi|\to\infty}
|\xi|^\alpha
\left|
\widehat{\mathcal H^1|_{\Gamma_0}}(\xi)
\right|
=
\infty.
\]
We also give a convex example for which arclength on every nontrivial subarc
fails even to be a Rajchman measure. The geometric mechanism behind these
examples is a positive curved trace: if $\Gamma$ contains a positive-length
subset of a $C^2$ curve whose curvature is bounded away from zero, then
$|E+\Gamma|>0$ whenever $\dim_{\mathrm H}(E)>1$. For a nondegenerate graph
this gives $T(\Gamma)=1$. For convex graphs it implies, in particular, that
$T(\Gamma)=1$ whenever the curvature measure has a nonzero absolutely
continuous part. The positive-measure proofs are in physical space and use
translated-tube intersections and elementary facts about convex functions.
The same overlap estimates give Mattila-type lower bounds for the average
lengths of the associated curve projections of neighborhoods under the
positive curved-trace hypothesis. We also prove a dimension-one endpoint
result for sets with a positive-length rectifiable part and formulate the main
remaining question: whether every strictly convex Lipschitz graph has the
optimal threshold $T(\Gamma)=1$.
\end{abstract}

\maketitle

\setcounter{tocdepth}{1}
\tableofcontents
\enlargethispage{4pt}

\section{Introduction}

Let $E\subset\mathbb R^2$ and let $\Gamma\subset\mathbb R^2$ be a compact
curve. We study the arithmetic sum
\begin{equation}\label{eq:intro-sum}
E+\Gamma
=
\{x+y:x\in E,\ y\in\Gamma\}.
\end{equation}
The question throughout the paper is whether geometric information about
$\Gamma$, together with a lower bound for $\dim_{\mathrm H}(E)$, forces the
set in \eqref{eq:intro-sum} to have positive planar Lebesgue measure.

For a compact curve $\Gamma$, define
\begin{equation}\label{eq:intro-threshold}
T(\Gamma)
=
\inf
\left\{
 t:
 \dim_{\mathrm H}(E)>t
 \Longrightarrow
 |E+\Gamma|>0
 \text{ for every compact }E\subset\mathbb R^2
\right\}.
\end{equation}
Throughout the paper, a nondegenerate graph means the graph of a continuous function
over a compact interval of positive length. For such a graph one always has
\begin{equation}\label{eq:intro-lower-threshold}
T(\Gamma)\geq1.
\end{equation}
Indeed, there is a compact set $B\subset\mathbb R$ with
$|B|=0$ and $\dim_{\mathrm H}(B)=1$. If
$E=\{0\}\times B$, then every vertical section of $E+\Gamma$ is a translate
of $B$, and Fubini's theorem gives $|E+\Gamma|=0$. Thus the strongest
possible universal conclusion for a graph is
\[
T(\Gamma)=1.
\]
Theorem \ref{thm:intro-curved-trace} below shows that this lower bound is
attained whenever the graph has a positive curved trace.

We use the following notation throughout. For nonnegative quantities $X$ and
$Y$, the relation $X\lesssim Y$ means that $X\leq CY$ for a constant $C>0$
independent of the variables being estimated. We write $X\gtrsim Y$ if
$Y\lesssim X$, and $X\asymp Y$ if both $X\lesssim Y$ and $Y\lesssim X$.
A subscript on any of these symbols indicates permitted dependence of the
implicit constant. The notation $O(Y)$ has the same upper-bound meaning, and
$o(Y)$ denotes a quantity whose ratio to $Y$ tends to zero in the indicated
limit. For a measurable set $A$, the notation $|A|$ denotes Lebesgue measure
in the ambient dimension, while $|x|$ denotes the Euclidean norm of a vector
$x$. For $A\subset\mathbb R^2$ and $\delta>0$, we write
\[
A^\delta
=
\{x\in\mathbb R^2:\operatorname{dist}(x,A)\leq\delta\}.
\]
For a finite Borel measure $\nu$ on $\mathbb R^d$, our Fourier transform
convention is
\[
\widehat\nu(\xi)
=
\int e^{-2\pi i x\cdot\xi}\,d\nu(x).
\]
For $0<s<2$ and a finite positive Borel measure $\mu$ on $\mathbb R^2$, we
write
\[
I_s(\mu)
=
\iint |x-y|^{-s}\,d\mu(x)\,d\mu(y).
\]
We write $1_A$ for the indicator function of $A$, and
$\operatorname{Var}_I(f)$ and $\operatorname{osc}_I(f)$ for the total
variation and oscillation of $f$ on an interval $I$.

For later comparison with projection theory, let
\[
\Gamma
=
\{(t,\gamma(t)):t\in[a,b]\}
\]
be a compact graph. For a compact set $E\subset\mathbb R^2$ and
$\alpha\in\mathbb R$, define the associated curve projection by
\begin{equation}\label{eq:intro-curve-projection}
\Phi_\alpha^\Gamma(E)
=
\left\{
 x_2+\gamma(\alpha-x_1):
 (x_1,x_2)\in E,\ \alpha-x_1\in[a,b]
\right\}.
\end{equation}
This is precisely the vertical section of $E+\Gamma$ above $\alpha$. These
curve projections were introduced by Simon and Taylor in their study of
planar sumsets \cite{SimonTaylor}. We define the Favard curve length of $E$
associated to $\Gamma$ by
\begin{equation}\label{eq:intro-favard-curve-length}
\operatorname{Fav}_\Gamma(E)
=
\int_{\mathbb R}|\Phi_\alpha^\Gamma(E)|\,d\alpha.
\end{equation}
Fubini's theorem gives the exact identity
\begin{equation}\label{eq:intro-favard-sum-identity}
\operatorname{Fav}_\Gamma(E)
=
|E+\Gamma|.
\end{equation}
Davey and Taylor obtained a quantitative nonlinear Besicovitch theorem for
this family \cite{DaveyTaylor}, while Cladek, Davey, and Taylor developed the
quantitative Favard curve problem for finite approximations of the four-corner
Cantor set \cite{CladekDaveyTaylor}. In the classical orthogonal setting,
Mattila proved lower bounds for the Favard lengths of neighborhoods with the
exponents $\delta^{1-s}$ for $0<s<1$ and
$\left(\log\frac{1}{\delta}\right)^{-1}$ for $s=1$
\cite{MattilaProjections}. Bongers and Taylor obtained the same exponents for
globally transversal nonlinear projection families; for curve projections,
their hypothesis when $d=2$ is that $\Gamma$ be piecewise $C^1$ with piecewise
bi-Lipschitz unit tangent \cite[Theorem~1.7]{BongersTaylor}. Corollary
\ref{cor:intro-quantitative-curve-projections} below gives these lower bounds when
$\Gamma$ merely has a positive curved trace. The trace may be nowhere dense,
and the ambient graph need not satisfy the global transversality hypothesis.

For the unit circle, the threshold-one conclusion is a special case of the
classical circle-packing problem. Let $I\subset(0,\infty)$ be a compact
interval and let $F\subset\mathbb R^2\times I$ be a Borel set of center-radius
pairs. Wolff proved that
\[
\dim_{\mathrm H}(F)>1
\quad\Longrightarrow\quad
\left|
\bigcup_{(a,r)\in F}
\{x\in\mathbb R^2:|x-a|=r\}
\right|>0;
\]
see \cite[Corollary~3]{WolffLocalSmoothing}. Taking
$F=E\times\{1\}$ gives
$|E+\mathbb S^1|>0$ whenever $\dim_{\mathrm H}(E)>1$. Earlier, Marstrand
proved the positive-measure conclusion when the set of centers has positive
planar measure, while Mitsis proved it in the plane when the set of centers has
Hausdorff dimension greater than $\frac32$, as well as for spheres in
dimensions at least three when the set of centers has Hausdorff dimension
greater than $1$; see \cite{MarstrandPacking,Mitsis}. Simon and Taylor gave a
nonlinear-projection treatment for fixed simple piecewise $C^2$ planar curves
with nonvanishing curvature and obtained the sharp dimension-one endpoint
characterization \cite{SimonTaylor}.

A related but distinct circular Kakeya problem asks how small a planar set can
be if it contains a circle of every radius $r\in[1,2]$. Positive Lebesgue
measure is false in this setting: Besicovitch and Rado, and independently
Kinney, constructed measure-zero examples \cite{BesicovitchRado,Kinney}. Wolff
nevertheless proved that every such set has Hausdorff dimension two
\cite{WolffCircularKakeya}. This radius-parametrized problem should be
distinguished from the center-parametrized fixed-translate problem studied
here.

In recent work, Iosevich, Li, and Taylor, three of the present authors, placed
this fixed-translate question in the broader setting of unions of smooth
variable-coefficient hypersurfaces \cite{IosevichLiTaylor}. At a fixed level,
they proved under the Phong--Stein rotational curvature condition that a
parameter set of Hausdorff dimension greater than one generates a union of
positive Lebesgue measure. In the planar translation-invariant case
$\Sigma_x=x+\Gamma$, this gives $T(\Gamma)=1$ when $\Gamma$ is a $C^2$ curve
with nonvanishing curvature. They also proved a geometric criterion based on
pairwise overlaps of tubular neighborhoods which reaches the same threshold
for arbitrary measurable selections. The present paper develops this
geometric mechanism in a substantially less regular setting. The curve
$\Gamma$ need not itself be $C^2$, and arclength on every nontrivial subarc may
fail all polynomial pointwise Fourier decay. What remains is a positive-length
trace on a uniformly curved graph, and we show directly that this weaker
feature supplies enough translated-neighborhood transversality to recover the
optimal threshold.

A standard proof in the smooth translation-invariant setting may be expressed
in terms of pointwise Fourier decay. If $\sigma$ is a nonzero finite positive
measure supported on a subarc of $\Gamma$ and, for some $0<\alpha<1$,
\begin{equation}\label{eq:intro-decay}
|\widehat\sigma(\xi)|
\lesssim
(1+|\xi|)^{-\alpha},
\end{equation}
then the Riesz-energy identity and Plancherel's theorem give
\begin{equation}\label{eq:intro-fourier-criterion}
\dim_{\mathrm H}(E)+2\alpha>2
\quad\Longrightarrow\quad
|E+\Gamma|>0.
\end{equation}
Indeed, put $s=2-2\alpha$ and choose
$s<\beta<\dim_{\mathrm H}(E)$. Frostman's lemma \cite{Mattila} gives a
nonzero $\beta$-Frostman measure $\mu$ supported on $E$. Summing over dyadic
annuli gives $I_s(\mu)<\infty$. Since
$|\widehat\sigma(\xi)|^2\lesssim(1+|\xi|)^{-2\alpha}$, the Riesz-energy
identity \cite{Mattila} shows that
$\widehat\mu\,\widehat\sigma\in L^2(\mathbb R^2)$.
Hence $\mu*\sigma$ has a nonzero $L^2$ density supported on $E+\Gamma$, and
therefore $|E+\Gamma|>0$. The complete argument is recorded in Proposition
\ref{prop:direct-fourier}. Our question is whether this geometric conclusion
can persist when the natural arclength measure has no such decay.

Related recent work concerns averaged, rather than pointwise, Fourier
information for fractal measures on smooth curved graphs. Orponen, Puliatti,
and Py\"or\"al\"a obtain $L^p$ estimates for Frostman measures on the parabola
and corresponding lower bounds for the dimensions of their iterated sumsets,
while Yi establishes convolution-energy estimates for Frostman measures on
$C^2$ graphs with nonvanishing curvature and obtains sharp $L^6$ Fourier
estimates in the range $s\geq\frac23$; see
\cite{OrponenPuliattiPyorala,Yi}. These results are complementary to the
present paper. They concern averaged Fourier behavior and energy or dimension
growth for measures supported on smooth curves, whereas our separation
concerns positive Lebesgue measure of $E+\Gamma$ despite the failure of every
pointwise polynomial decay estimate for the natural arclength measure on every
nontrivial subarc of a low-regularity convex graph.

There are two points in this question which are worth separating. First, the
conclusion $|E+\Gamma|>0$ is a statement about the whole family of translates
of $\Gamma$ indexed by $E$, whereas \eqref{eq:intro-decay} is a pointwise
estimate for one fixed measure on the curve. It is therefore conceivable that
the translate geometry retains enough transversality even when the Fourier
transform of arclength behaves badly along a sequence of directions and
frequencies. Second, we are not asking merely whether one particular global
arclength measure fails to decay. Such a failure could be caused by a small
exceptional portion of the curve and would say little about the local
geometry. The examples below are arranged so that the Fourier obstruction
occurs on every nontrivial subarc. Thus the positive-measure conclusion cannot
be recovered by simply discarding a bad part of the curve and applying the
usual Fourier argument to a remaining arc.

The main example gives a strong separation.

\begin{theorem}[Strict convexity without polynomial decay]
\label{thm:intro-strict-example}
There is a strictly convex Lipschitz function
$\gamma_1:[1,2]\to\mathbb R$ whose graph $\Gamma_1$ satisfies
\[
T(\Gamma_1)=1.
\]
Nevertheless, for every nondegenerate interval $I\subset[1,2]$, every
$\alpha>0$, and
\[
\sigma_I
=
\mathcal H^1|_{\{(x,\gamma_1(x)):x\in I\}},
\]
one has
\begin{equation}\label{eq:intro-no-polynomial-decay}
\limsup_{|\xi|\to\infty}
|\xi|^\alpha|\widehat\sigma_I(\xi)|
=
\infty.
\end{equation}
In particular, no positive-length subarc of $\Gamma_1$ has pointwise Fourier
decay of any positive power.
\end{theorem}

Thus the optimal threshold $T(\Gamma)=1$ is not governed by the best
polynomial decay exponent of arclength on a subarc. The construction is
designed so that two apparently competing features coexist. On a
positive-length nowhere-dense set the curve agrees with a parabola, which
supplies geometric transversality. In every complementary interval, however,
the curve contains pieces that are almost flat at a superexponential
sequence of scales. These pieces force the failure of every polynomial
Fourier bound. Section \ref{sec:strict-example} gives the construction and
the complete oscillatory estimate.

The positive result used in this example is elementary and entirely
geometric. We isolate it in a form that is useful on its own.

\begin{definition}\label{def:intro-curved-trace}
A compact rectifiable curve $\Gamma$ has a positive curved trace if there
are a compact $C^2$ graph
\[
\Sigma=\{(x,g(x)):x\in I\},
\]
a compact set $A\subset\Gamma\cap\Sigma$, and a constant $c_0>0$ such that
\[
\mathcal H^1(A)>0
\qquad\text{and}\qquad
|g''(x)|\geq c_0
\]
for every $x\in I$.
\end{definition}

The set $A$ need not contain an arc. It may be nowhere dense in both
$\Gamma$ and $\Sigma$. Figure \ref{fig:positive-curved-trace} shows the
geometry schematically.

\begin{figure}[H]
\centering
\begin{tikzpicture}[x=1cm,y=1cm,>=Stealth]
  \draw[densely dashed,line width=.65pt]
    plot[samples=120,domain=-3.4:3.4] (\x,{0.105*\x*\x-.85});
  \node[font=\small] at (3.55,.28) {$\Sigma$};

  \draw[line width=.9pt]
    plot[samples=30,domain=-3.4:-2.65] (\x,{0.105*\x*\x-.85});
  \draw[line width=.9pt] (-2.65,{0.105*(-2.65)*(-2.65)-.85})
    -- (-2.05,{0.105*(-2.05)*(-2.05)-.85});
  \draw[line width=.9pt]
    plot[samples=30,domain=-2.05:-1.35] (\x,{0.105*\x*\x-.85});
  \draw[line width=.9pt] (-1.35,{0.105*(-1.35)*(-1.35)-.85})
    -- (-.85,{0.105*(-.85)*(-.85)-.85});
  \draw[line width=.9pt]
    plot[samples=30,domain=-.85:-.20] (\x,{0.105*\x*\x-.85});
  \draw[line width=.9pt] (-.20,{0.105*(-.20)*(-.20)-.85})
    -- (.28,{0.105*(.28)*(.28)-.85});
  \draw[line width=.9pt]
    plot[samples=30,domain=.28:1.00] (\x,{0.105*\x*\x-.85});
  \draw[line width=.9pt] (1.00,{0.105*(1.00)*(1.00)-.85})
    -- (1.48,{0.105*(1.48)*(1.48)-.85});
  \draw[line width=.9pt]
    plot[samples=30,domain=1.48:2.20] (\x,{0.105*\x*\x-.85});
  \draw[line width=.9pt] (2.20,{0.105*(2.20)*(2.20)-.85})
    -- (2.72,{0.105*(2.72)*(2.72)-.85});
  \draw[line width=.9pt]
    plot[samples=30,domain=2.72:3.4] (\x,{0.105*\x*\x-.85});

  \draw[line width=2.5pt]
    plot[samples=25,domain=-3.25:-2.82] (\x,{0.105*\x*\x-.85});
  \draw[line width=2.5pt]
    plot[samples=25,domain=-1.92:-1.50] (\x,{0.105*\x*\x-.85});
  \draw[line width=2.5pt]
    plot[samples=25,domain=-.72:-.34] (\x,{0.105*\x*\x-.85});
  \draw[line width=2.5pt]
    plot[samples=25,domain=.42:.86] (\x,{0.105*\x*\x-.85});
  \draw[line width=2.5pt]
    plot[samples=25,domain=1.62:2.05] (\x,{0.105*\x*\x-.85});
  \draw[line width=2.5pt]
    plot[samples=25,domain=2.86:3.25] (\x,{0.105*\x*\x-.85});

  \node[font=\small] at (-3.05,.72) {$\Gamma$};
  \node[font=\small,align=center] at (0,1.12)
    {positive-length trace $A\subset\Gamma\cap\Sigma$};
  \draw[->,line width=.5pt] (-.75,.92)--(-1.70,-.47);
  \draw[->,line width=.5pt] (.75,.92)--(1.83,-.49);
\end{tikzpicture}
\caption{A schematic positive curved trace. The dashed graph $\Sigma$ has
curvature bounded away from zero. The curve $\Gamma$ agrees with $\Sigma$ on
the thick portions representing a positive-length closed set $A$, but may
differ on every complementary interval. In the applications, $A$ may be
nowhere dense. Only the inclusion $E+A\subset E+\Gamma$ and the translated-
neighborhood geometry of $\Sigma$ are used.}
\label{fig:positive-curved-trace}
\end{figure}
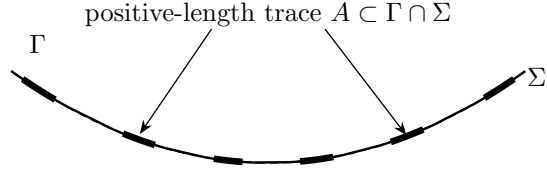

\begin{theorem}[Curved traces]\label{thm:intro-curved-trace}
Let $\Gamma\subset\mathbb R^2$ be a compact rectifiable curve having a
positive curved trace. If $E\subset\mathbb R^2$ is compact and
\[
\dim_{\mathrm H}(E)>1,
\]
then
\[
|E+\Gamma|>0.
\]
\end{theorem}

\begin{corollary}[Quantitative curve projections]
\label{cor:intro-quantitative-curve-projections}
Let $\Gamma$ be a nondegenerate compact rectifiable graph having a positive
curved trace. Let $E\subset\mathbb R^2$ be compact, and suppose that $E$
supports a Borel probability measure $\mu$ satisfying
\[
\mu(B(x,r))
\leq
C_\mu r^s
\]
for every $x\in\mathbb R^2$ and $0<r\leq1$, where $0<s\leq1$. Then there are
constants $c>0$ and $0<\delta_0<\frac12$, depending only on the curved
trace, $s$, and $C_\mu$, such that
\begin{equation}\label{eq:intro-quantitative-curve-projections}
\operatorname{Fav}_\Gamma(E^\delta)
\geq
c
\begin{cases}
\delta^{1-s},&0<s<1,\\[4pt]
\displaystyle\left(\log\frac{1}{\delta}\right)^{-1},&s=1,
\end{cases}
\end{equation}
for every $0<\delta<\delta_0$.
\end{corollary}

By \eqref{eq:intro-favard-curve-length}, this is a lower bound for the average
length of the curve projections $\Phi_\alpha^\Gamma(E^\delta)$. The conclusion
is therefore averaged in $\alpha$; it is distinct from the two-projection
statement used later at the rectifiable endpoint. The point here is that the
standard Mattila exponents persist even though the full projection family need
not be globally transversal. In particular, the corollary applies to the
strictly convex graph in Theorem \ref{thm:intro-strict-example}: it has the
usual quantitative neighborhood projection bounds even though arclength on
every nontrivial subarc fails polynomial decay of every positive power.

Combining Theorem \ref{thm:intro-curved-trace} with the lower bound
\eqref{eq:intro-lower-threshold}, we obtain $T(\Gamma)=1$ whenever $\Gamma$
is a nondegenerate graph with a positive curved trace.

Here is the idea of the proof; the complete argument is given in Section
\ref{sec:overlap}. A uniformly curved graph satisfies
\begin{equation}\label{eq:intro-overlap}
|\Sigma^\delta\cap(h+\Sigma^\delta)|
\lesssim
\frac{\delta^2}{\delta+|h|}.
\end{equation}
Since $A\subset\Sigma$, the same upper bound holds with $A$ in place of
$\Sigma$, while the positive length of $A$ gives $|A^\delta|\asymp\delta$.
A second-moment argument applied to the family of translates
$\{x+A^\delta:x\in E\}$ then turns the overlap estimate into a lower bound for
$|E+A^\delta|$ that is uniform in $\delta$. Passing to the limit gives
$|E+A|>0$, and hence $|E+\Gamma|>0$. No Fourier transform of $A$ or $\Gamma$
enters the proof.

For a convex graph, this has an intrinsic consequence. Let
$\gamma:[a,b]\to\mathbb R$ be convex. Its distributional second derivative
is a finite positive measure. By the Lebesgue decomposition theorem, we write
\begin{equation}\label{eq:intro-curvature-decomposition}
D^2\gamma
=
k(x)\,dx+\kappa_s,
\end{equation}
where $k(x)\,dx$ is the absolutely continuous part and $\kappa_s$ is singular
with respect to Lebesgue measure.

\begin{theorem}[Absolutely continuous curvature]\label{thm:intro-ac-curvature}
Let $\Gamma$ be the graph of a convex Lipschitz function $\gamma$ on a
compact interval. If
\begin{equation}\label{eq:intro-positive-ac-curvature}
\left|\{x:k(x)>0\}\right|>0,
\end{equation}
where $k$ is the density in
\eqref{eq:intro-curvature-decomposition}, then
\[
T(\Gamma)=1.
\]
\end{theorem}

The proof combines Theorem \ref{thm:intro-curved-trace} with the
$C^2$-Lusin approximation theorem for one-dimensional convex functions due
to Goldstein and Haj{\l}asz \cite{GoldsteinHajlasz}. Thus a nonzero
absolutely continuous part of the curvature measure already forces the
optimal universal threshold. The point is not that it supplies a convenient
route to Fourier decay; it supplies a positive curved trace.

This also explains why the Lebesgue decomposition of the curvature measure is
natural for the present problem, but the singular part must be treated with
some care. The absolutely continuous part can be detected on a
positive-measure set of parameters and, through the $C^2$-Lusin
approximation, converted into agreement with a genuinely curved $C^2$ graph
on a positive-measure set. Atomicity alone, however, does not determine the
geometry. A finitely supported atomic curvature measure produces polygonal
behavior, whereas a purely atomic measure whose atoms are dense may make the
slope strictly increasing and its primitive strictly convex. The
Cantor-quantile pieces used in Section \ref{sec:strict-example} have precisely
this feature; see Remark \ref{rem:q-atomic}. Purely singular continuous
curvature is a second singular regime. It may also force the slope to change
on every interval while remaining concentrated on a Lebesgue-null set. The
argument producing a positive curved trace does not treat either singular
regime in general.

Before giving the strictly convex construction, we record a simpler example
in which Fourier decay fails completely.

\begin{theorem}[No decay on any subarc]\label{thm:intro-nonrajchman}
There is a convex Lipschitz function $\gamma_0:[1,2]\to\mathbb R$ whose graph
$\Gamma_0$ satisfies
\[
T(\Gamma_0)=1,
\]
but for every nondegenerate interval $I\subset[1,2]$, if
\[
\sigma_I
=
\mathcal H^1|_{\{(x,\gamma_0(x)):x\in I\}},
\]
then $\sigma_I$ is not a Rajchman measure. More precisely, there is a unit
vector $\omega_I$ such that
\[
\widehat\sigma_I(R\omega_I)
\not\longrightarrow0
\qquad\text{as }R\to\infty.
\]
\end{theorem}

The curve agrees with the parabola $y=\frac{x^2}{2}$ over a fat Cantor set
and is affine on every complementary interval. Its intersection with the
parabola has positive length, so the curved-trace theorem applies. Every
nontrivial subarc also contains a line segment. Projection in the normal
direction of that segment produces an atom, and Wiener's mean-square identity
rules out decay. The strict example replaces these line segments by carefully
chosen strictly convex pieces while retaining the same positive curved trace.

There is one further positive-measure statement which is useful at the
dimension-one endpoint and requires almost no curvature.

\begin{theorem}[Rectifiable endpoint]\label{thm:intro-rectifiable}
Let $\Gamma\subset\mathbb R^2$ be a compact rectifiable curve which is not
contained in a line. Suppose that $E\subset\mathbb R^2$ contains a countably
$1$-rectifiable set $R$ satisfying
\[
\mathcal H^1(R)>0.
\]
Then there are compact sets
\[
K_+\subset E+\Gamma
\qquad\text{and}\qquad
K_-\subset E-\Gamma
\]
such that $|K_+|>0$ and $|K_-|>0$.
\end{theorem}

This follows directly from the area formula applied to
$(s,t)\mapsto e(s)+r(t)$, where $e$ and $r$ parametrize rectifiable pieces of
$E$ and $\Gamma$. In the special case where $\Gamma$ is the graph of a $C^1$
function with injective derivative, the conclusion can also be deduced from
the nonlinear two-projection theorem of Li and Taylor
\cite[Corollary~2.5]{LiTaylorTwoProjection}. Indeed, after restricting $R$ to
a compact positive-length subset whose horizontal diameter is smaller than
the length of the parameter interval, one may extend the graph function to a
$C^1$ function on $\mathbb R$ with injective derivative. On a nondegenerate
interval of section parameters, the corresponding curve projections are then
precisely the vertical sections of $R+\Gamma$. The contrapositive of their
corollary says that at most one of these sections can have zero length, and
Fubini's theorem gives $|R+\Gamma|>0$. Applying the same argument to
$-\Gamma$ gives the difference-set conclusion. This places this special case
of Theorem \ref{thm:intro-rectifiable} within the framework of the
two-projection theorem and its nonlinear analogues.

The variable-surface result of Iosevich, Li, and Taylor
\cite{IosevichLiTaylor} also contains a fixed-level endpoint theorem for
positive-length $1$-rectifiable parameter sets under rotational curvature. In
the planar translation-invariant setting, Theorem
\ref{thm:intro-rectifiable} shows that curvature can be dispensed with
entirely: it is enough that $\Gamma$ be rectifiable and not contained in a
line. For smooth curves of nonvanishing curvature, Simon and Taylor proved a
sharper characterization at dimension one in terms of pure unrectifiability
\cite{SimonTaylor}. The theorem above records the positive direction under
almost no curvature assumption.

For convex graphs, the decomposition
\eqref{eq:intro-curvature-decomposition} suggests several distinct regimes.
A nonzero absolutely continuous component gives $T(\Gamma)=1$. A finitely
supported atomic curvature measure produces polygonal behavior, but atomic
curvature with dense support may instead produce a strictly convex graph. Atomicity by
itself therefore does not determine the threshold. Purely singular continuous
curvature gives another natural unresolved regime. Such a curve may be
strictly convex and, by Proposition \ref{prop:singular-no-trace}, contains no
positive curved trace. These observations lead to the sharper
question of whether strict convexity alone forces
\[
T(\Gamma)=1.
\]

We emphasize the sharpness and the locality of the results. The lower bound
$T(\Gamma)\geq1$ uses only that $\Gamma$ is a graph over an interval of
positive length, and therefore none of our positive results can improve the
dimension threshold. The curved-trace theorem reaches this lower bound from a
hypothesis which may hold on a very small topological portion of the curve:
the trace can be nowhere dense and need not contain any subarc. What matters
is that it has positive length and lies on a uniformly curved $C^2$ graph.

The Fourier statements are deliberately local in the opposite direction. In
Theorems \ref{thm:intro-strict-example} and \ref{thm:intro-nonrajchman}, the
bad Fourier behavior is required on every nontrivial subarc. This prevents the
optimal threshold from being explained by selecting some other ordinary
subarc on which arclength has the familiar curvature decay. The two examples
serve different purposes. The non-Rajchman example gives the strongest
failure of Fourier decay, but it contains line segments. The strictly convex
example removes those segments and shows that strict convexity alone does not
restore any polynomial arclength decay. In both cases the positive curved
trace is what retains the sharp Minkowski-sum conclusion.

The endpoint theorem has a different role. It does not strengthen the
universal statement for arbitrary sets of Hausdorff dimension one. Instead,
it shows that once the translating set itself contains positive-length
rectifiable structure, very little curvature is needed from $\Gamma$: the
curve need only fail to lie in a line. We include this result because it
clarifies which part of the phenomenon belongs to the geometry of the curve
and which part belongs to the geometry of the translating set.

The overlap framework used here builds on the variable-surface setting of
\cite{IosevichLiTaylor}, but the present paper develops a translation-invariant
version adapted to convex curves of very low regularity. In particular, the
argument requires neither smoothness of the entire curve nor pointwise Fourier
decay of arclength; the essential geometric input is the presence of a
positive curved trace.

Section \ref{sec:overlap} proves the translated-tube principle, the
positive curved-trace theorem, and the quantitative curve-projection
corollary. Section \ref{sec:curvature-measure} treats
the absolutely continuous part of the curvature measure. Section
\ref{sec:nonrajchman} gives the convex non-Rajchman example. Section
\ref{sec:strict-example} constructs the strictly convex example and proves
the failure of polynomial Fourier decay on every subarc. Section
\ref{sec:rectifiable} proves the rectifiable endpoint theorem. Section
\ref{sec:fourier-comparison} records the direct Fourier benchmark,
gives a quantitative result for singular curvature, and formulates the
remaining strict-convexity question.

\section{Translated tubes and positive curved traces}\label{sec:overlap}

As above, for a compact set $A\subset\mathbb R^2$ and $\delta>0$, we write
\begin{equation}\label{eq:tube-definition}
A^\delta
=
\{x\in\mathbb R^2:\operatorname{dist}(x,A)\leq\delta\}.
\end{equation}
\begin{proposition}\label{prop:graph-threshold-lower}
Let $\Gamma$ be the graph of a continuous function on a nondegenerate compact
interval. Then
\[
T(\Gamma)\geq1.
\]
\end{proposition}

\begin{proof}
Choose a compact set $B\subset[0,1]$ with $|B|=0$ and
$\dim_{\mathrm H}(B)=1$, and put $E=\{0\}\times B$. If
\[
\Gamma=\{(x,\gamma(x)):x\in[a,b]\},
\]
then the vertical section of $E+\Gamma$ above $x\in[a,b]$ is
$\gamma(x)+B$. It has one-dimensional Lebesgue measure zero. Fubini's
theorem gives $|E+\Gamma|=0$, while $\dim_{\mathrm H}(E)=1$.
\end{proof}

We next isolate the physical-space argument which turns a
translated-neighborhood estimate into a positive-measure theorem. This is the
translation-invariant form of the geometric intersection mechanism used in
\cite{IosevichLiTaylor}. Closely related intersection estimates for translates
of smooth curved tubes appear in \cite[Section~4]{Yi}. Our purpose here is to
formulate the precise uniform area estimate needed for the second-moment
argument and then to apply it to positive-length traces inside curves of very
low regularity.

\begin{theorem}[Overlap-to-coverage principle]\label{thm:overlap-coverage}
Let $A\subset\mathbb R^2$ be compact. Suppose that there are constants
$c,C>0$ and $\delta_0>0$ such that
\begin{equation}\label{eq:tube-volume-hypothesis}
|A^\delta|
\geq
c\delta
\end{equation}
for $0<\delta<\delta_0$, and
\begin{equation}\label{eq:tube-overlap-hypothesis}
|A^\delta\cap(h+A^\delta)|
\leq
C\frac{\delta^2}{\delta+|h|}
\end{equation}
for every $h\in\mathbb R^2$ and $0<\delta<\delta_0$.

Let $E\subset\mathbb R^2$ be compact, let $\mu$ be a nonzero finite positive
measure supported on $E$, and, for $0<\delta<\delta_0$, put
\begin{equation}\label{eq:regularized-one-energy}
J_\delta(\mu)
=
\iint
\frac{d\mu(x)d\mu(y)}{\delta+|x-y|}.
\end{equation}
Then
\begin{equation}\label{eq:regularized-overlap-coverage}
|E+A^\delta|
\geq
c'
\frac{\mu(E)^2}{J_\delta(\mu)}.
\end{equation}
If $I_1(\mu)<\infty$, then
\begin{equation}\label{eq:quantitative-overlap-coverage}
|E+A|
\geq
c'
\frac{\mu(E)^2}{I_1(\mu)}.
\end{equation}
Consequently,
\[
\dim_{\mathrm H}(E)>1
\quad\Longrightarrow\quad
|E+A|>0.
\]
The constant $c'>0$ depends only on the constants in
\eqref{eq:tube-volume-hypothesis} and
\eqref{eq:tube-overlap-hypothesis}.
\end{theorem}

Notice that the overlap hypothesis with $h=0$ already gives
$|A^\delta|\leq C\delta$. Thus the lower bound in
\eqref{eq:tube-volume-hypothesis} is the only separate tube-volume assumption.

\begin{proof}
For $0<\delta<\delta_0$, define
\begin{equation}\label{eq:Fdelta}
F_\delta(z)
=
\int_E
\frac{1_{x+A^\delta}(z)}{|A^\delta|}
d\mu(x).
\end{equation}
The function $F_\delta$ is supported on $E+A^\delta$ and
\begin{equation}\label{eq:Fdelta-mass}
\int F_\delta(z)dz
=
\mu(E).
\end{equation}
Fubini's theorem gives
\begin{align}
\|F_\delta\|_2^2
&=
\iint
\frac{|A^\delta\cap((y-x)+A^\delta)|}{|A^\delta|^2}
d\mu(x)d\mu(y)
\notag\\
&\leq
C_0
\iint
\frac{d\mu(x)d\mu(y)}{\delta+|x-y|}
\notag\\
&=
C_0J_\delta(\mu)
<
\infty.
\label{eq:Fdelta-energy}
\end{align}
By Cauchy--Schwarz and \eqref{eq:Fdelta-mass},
\[
\mu(E)^2
\leq
|E+A^\delta|\|F_\delta\|_2^2.
\]
Equation \eqref{eq:Fdelta-energy} proves
\eqref{eq:regularized-overlap-coverage}.

Suppose now that $I_1(\mu)<\infty$. This assumption implies that $\mu$ has
no atoms, and hence the diagonal has $\mu\times\mu$ measure zero. Since
\[
\frac{1}{\delta+|x-y|}
\leq
\frac{1}{|x-y|}
\]
away from the diagonal, it follows that $J_\delta(\mu)\leq I_1(\mu)$. Thus
\eqref{eq:regularized-overlap-coverage} gives
\[
|E+A^\delta|
\geq
\frac{\mu(E)^2}{C_0I_1(\mu)},
\]
a lower bound independent of $\delta$. Choose a decreasing sequence
$0<\delta_j<\delta_0$ with $\delta_j\downarrow0$. The compact
sets $E+A^{\delta_j}$ decrease, and
\begin{equation}\label{eq:intersection-thickened-sums}
\bigcap_j(E+A^{\delta_j})
=
E+A.
\end{equation}
Indeed, if $z=x_j+y_j$ belongs to every set on the left, with $x_j\in E$ and
$\operatorname{dist}(y_j,A)\leq\delta_j$, then $x_j$ lies in the compact set
$E$ and $y_j$ lies in the compact set $A^{\delta_1}$. Passing to a common
subsequence, we may suppose that $x_j\to x\in E$ and $y_j\to y$. Since
$\operatorname{dist}(y_j,A)\to0$ and $A$ is closed, one has $y\in A$.
Thus $z=x+y$. The reverse
inclusion in \eqref{eq:intersection-thickened-sums} is immediate. Continuity
from above of Lebesgue measure now proves
\eqref{eq:quantitative-overlap-coverage}.

If $\dim_{\mathrm H}(E)>1$, choose $s$ with
$1<s<\dim_{\mathrm H}(E)$. Frostman's lemma gives a nonzero
$s$-Frostman measure $\mu$ on $E$. Thus
$\mu(B(x,r))\leq C_\mu r^s$. Splitting the region $|x-y|<1$ into the
annuli
\[
2^{-j-1}\leq |x-y|<2^{-j},
\qquad j\geq0,
\]
and using the trivial bound $|x-y|^{-1}\leq1$ when $|x-y|\geq1$, we obtain
\[
I_1(\mu)
\leq
\mu(E)^2
+
C\mu(E)\sum_{j=0}^\infty 2^{-j(s-1)}
<
\infty.
\]
The quantitative statement applies.
\end{proof}

The exponent in the overlap estimate is exactly matched to the $1$-energy of
the translating measure. When $|h|$ is large compared with $\delta$, two
curved $\delta$-tubes overlap in area $O(\delta^2/|h|)$; after normalizing each
translate by
$|A^\delta|\asymp\delta$, the resulting second moment is controlled by the
kernel $|x-y|^{-1}$. This is the physical-space counterpart of the
half-derivative of smoothing that appears in the classical Fourier argument
for planar curves. The proof above needs only this averaged overlap control
and does not require pointwise information about the Fourier transform of
arclength.

Figure \ref{fig:translated-tube-overlap} shows the geometry behind the
estimate. When the horizontal component of the translation is $u$, uniform
curvature forces the slopes of the two graphs at a potential intersection to
differ by a quantity comparable to $|u|$. Thus, once $|h|$ is larger than the
tube thickness, the two tubes can remain close only over a horizontal interval
of length $O(\delta/|h|)$. Multiplying by the transverse thickness
$O(\delta)$ gives the required area bound.

\begin{figure}[H]
\centering
\begin{tikzpicture}[x=1cm,y=1cm,>=Stealth]
  \draw[line width=.85pt]
    plot[samples=120,domain=-3.4:3.4] (\x,{0.09*\x*\x-.55});
  \draw[densely dashed,line width=.7pt]
    plot[samples=120,domain=-2.8:4.0] (\x,{0.09*(\x-0.65)*(\x-0.65)-.12});

  \draw[line width=.45pt]
    plot[samples=120,domain=-3.4:3.4] (\x,{0.09*\x*\x-.40});
  \draw[line width=.45pt]
    plot[samples=120,domain=-3.4:3.4] (\x,{0.09*\x*\x-.70});

  \draw[densely dashed,line width=.4pt]
    plot[samples=120,domain=-2.8:4.0] (\x,{0.09*(\x-0.65)*(\x-0.65)+.03});
  \draw[densely dashed,line width=.4pt]
    plot[samples=120,domain=-2.8:4.0] (\x,{0.09*(\x-0.65)*(\x-0.65)-.27});

  \draw[->,line width=.65pt] (-2.55,-.30)--(-1.90,.13)
    node[midway,above left,font=\small] {$h$};

  \draw[<->,line width=.5pt] (1.05,-.44)--(1.05,-.14)
    node[midway,right,font=\scriptsize] {$\delta$};

  \draw[<->,line width=.5pt] (-.20,-.36)--(.55,-.36)
    node[midway,below=2pt,font=\scriptsize] {$O(\delta/|h|)$};

  \draw[rounded corners=1pt,line width=.65pt]
    (-.28,-.52) rectangle (.62,-.03);

  \node[font=\small] at (-2.55,.90) {$\Sigma^\delta$};
  \node[font=\small] at (2.95,1.02) {$h+\Sigma^\delta$};
  \node[font=\scriptsize,align=center] at (1.35,.52)
    {overlap region\\$O(\delta^2/|h|)$};
  \draw[->,line width=.45pt] (1.05,.40)--(.48,-.02);
\end{tikzpicture}
\caption{Schematic translated-tube geometry. Uniform curvature makes the
relative slope of the two translated graphs grow linearly with the horizontal
part of the translation. Away from the regime $|h|\lesssim\delta$, the
intersection therefore has longitudinal size $O(\delta/|h|)$ and transverse
size $O(\delta)$.}
\label{fig:translated-tube-overlap}
\end{figure}
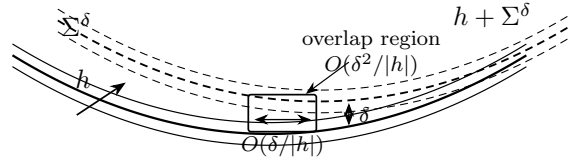

We now prove the overlap estimate for a uniformly curved graph.

\begin{theorem}[Uniformly curved graphs]\label{thm:curved-graph-overlap}
Let $g\in C^2([a,b])$ and suppose that
\begin{equation}\label{eq:curvature-lower}
|g''(x)|\geq c_0>0
\end{equation}
for every $x\in[a,b]$. Let
\[
\Sigma
=
\{(x,g(x)):a\leq x\leq b\}.
\]
There are constants $C>0$ and $\delta_0>0$ such that
\begin{equation}\label{eq:curved-graph-overlap}
|\Sigma^\delta\cap(h+\Sigma^\delta)|
\leq
C\frac{\delta^2}{\delta+|h|}
\end{equation}
for every $h\in\mathbb R^2$ and $0<\delta<\delta_0$.
\end{theorem}

\begin{proof}
Since $g''$ is continuous and never vanishes, it has constant sign. After
replacing $g$ by $-g$ if necessary, suppose that $g''\geq c_0$. Write
$h=(u,v)$ and let $L=\|g'\|_\infty$. Throughout the proof we use
\begin{equation}\label{eq:sum-max-comparison}
\max\{\delta,|h|\}
\leq
\delta+|h|
\leq
2\max\{\delta,|h|\}.
\end{equation}
Thus the right side of \eqref{eq:curved-graph-overlap} is comparable to
$\delta$ when $|h|\lesssim\delta$ and to $\delta^2/|h|$ when
$|h|\gtrsim\delta$. These are the two regimes treated below.

The Euclidean $\delta$-neighborhood of $\Sigma$ is contained in
\[
V_\delta
=
\{(x,y):a\leq x\leq b,\ |y-g(x)|\leq C_0\delta\},
\]
together with the two disks
\[
D_{a,\delta}
=
B((a,g(a)),C_0\delta),
\qquad
D_{b,\delta}
=
B((b,g(b)),C_0\delta),
\]
where $C_0$ depends only on $L$. Indeed, if $(x,y)$ is within $\delta$ of
$(t,g(t))\in\Sigma$ and $x\in[a,b]$, then
\[
|y-g(x)|
\leq
|y-g(t)|+L|x-t|
\leq
(1+L)\delta.
\]
If $x<a$ or $x>b$, then $t$ lies within $\delta$ of the corresponding
endpoint, and $(x,y)$ lies in one of the two displayed disks after increasing
$C_0$. We first estimate the overlap of the two vertical neighborhoods.

Suppose that a point of $V_\delta\cap(h+V_\delta)$ has horizontal coordinate
$x$. Then $x$ and $x-u$ both lie in $[a,b]$, and
\begin{equation}\label{eq:curved-sublevel}
|g(x)-g(x-u)-v|
\leq
2C_0\delta.
\end{equation}
Assume first that $u>0$. On the interval on which both terms are defined, the
function
\[
G_u(x)=g(x)-g(x-u)
\]
satisfies
\[
G_u'(x)
=
g'(x)-g'(x-u)
=
\int_{x-u}^x g''(t)dt
\geq
c_0u.
\]
Consequently, if $x_1<x_2$ both satisfy \eqref{eq:curved-sublevel}, then
\[
c_0u(x_2-x_1)
\leq
G_u(x_2)-G_u(x_1)
\leq
4C_0\delta.
\]
The set of possible horizontal coordinates therefore has length at most
$C\delta/u$. The same conclusion holds for $u<0$, with $|u|$ in place of
$u$. When $u=0$, we use the trivial bound by $b-a$. Since every vertical
section of the overlap has length at most $2C_0\delta$, we obtain
\begin{equation}\label{eq:vertical-curved-overlap}
|V_\delta\cap(h+V_\delta)|
\leq
C\delta
\min
\left\{
1,
\frac{\delta}{|u|}
\right\},
\end{equation}
where the minimum is understood to be $1$ when $u=0$.

If this overlap is nonempty, \eqref{eq:curved-sublevel} and the Lipschitz
property give
\[
|v|
\leq
|g(x)-g(x-u)|+2C_0\delta
\leq
L|u|+2C_0\delta.
\]
It follows that
\[
|h|
\leq
|u|+|v|
\leq
(1+L)|u|+2C_0\delta.
\]
Choose $C_1>4C_0$. If $|h|\geq C_1\delta$ and the overlap is nonempty, then
\[
|u|
\geq
\frac{|h|}{2(1+L)}.
\]
Equation \eqref{eq:vertical-curved-overlap} therefore gives
\[
|V_\delta\cap(h+V_\delta)|
\leq
C\frac{\delta^2}{|h|}.
\]
If $|h|<C_1\delta$, the trivial estimate
$|V_\delta\cap(h+V_\delta)|\leq |V_\delta|\leq C\delta$ is the right side of
\eqref{eq:curved-graph-overlap}, up to a constant, by
\eqref{eq:sum-max-comparison}.

It remains to estimate the intersections in which at least one of the two
sets is an endpoint disk. Each such intersection is contained in a disk of
radius $C_0\delta$, and hence has area at most $C\delta^2$. Moreover, every
point of $V_\delta\cup D_{a,\delta}\cup D_{b,\delta}$ lies within
$C_0\delta$ of $\Sigma$. If one of the endpoint-disk intersections is
nonempty, there are points $p,q\in\Sigma$ such that
\[
|h-(p-q)|
\leq
2C_0\delta.
\]
Thus
\[
|h|
\leq
\operatorname{diam}(\Sigma)+2C_0\delta.
\]
After decreasing $\delta_0$ so that $\delta_0\leq1$, this implies
$\delta+|h|\leq C_\Sigma$. Hence
\[
\delta^2
\leq
C_\Sigma\frac{\delta^2}{\delta+|h|}.
\]
There are only finitely many endpoint-disk intersections, so their total
contribution satisfies the required bound. Together with the estimate for the
vertical neighborhoods, this proves \eqref{eq:curved-graph-overlap}.
\end{proof}

The next elementary volume estimate allows us to pass from a curved graph to
a positive-length subset of it.

\begin{lemma}\label{lem:trace-tube-volume}
Let $g$ be Lipschitz on $[a,b]$, let $S\subset[a,b]$ be compact with
$|S|>0$, and put
\[
A
=
\{(x,g(x)):x\in S\}.
\]
Then there are constants $c,C>0$ such that
\[
c\delta
\leq
|A^\delta|
\leq
C\delta
\]
for every $0<\delta\leq1$.
\end{lemma}

\begin{proof}
Let $L$ be a Lipschitz constant for $g$. If a point $(x,y)$ lies within
$\delta$ of the full graph and $x\in[a,b]$, then
\[
|y-g(x)|
\leq
(1+L)\delta.
\]
If $x$ lies outside $[a,b]$, the point lies in a $C\delta$-disk about one of
the two endpoints of the graph. Thus the $\delta$-neighborhood of the full
graph is contained in
\[
\{(x,y):a\leq x\leq b,\ |y-g(x)|\leq(1+L)\delta\}
\]
together with two disks of radius $C\delta$. Its area is at most
\[
2(1+L)(b-a)\delta+2\pi C^2\delta^2
\leq
C'\delta
\]
for $0<\delta\leq1$. Since $A$ is contained in the full graph, this proves
the upper bound.

For the lower bound, if $x\in S$, then the vertical section of $A^\delta$
above $x$ contains the interval
\[
[g(x)-\delta,g(x)+\delta].
\]
Fubini's theorem therefore gives
\[
|A^\delta|
\geq
2|S|\delta.
\]
\end{proof}

\begin{theorem}[Positive curved traces]\label{thm:positive-curved-trace}
Let $\Gamma\subset\mathbb R^2$ be a compact rectifiable curve. Suppose that
there are a compact $C^2$ graph
\[
\Sigma=\{(x,g(x)):x\in[a,b]\},
\]
a compact set $S\subset[a,b]$ of positive Lebesgue measure, and
\begin{equation}\label{eq:trace-inclusion}
A
=
\{(x,g(x)):x\in S\}
\subset
\Gamma
\end{equation}
such that $|g''|\geq c_0>0$ on $[a,b]$. If $E\subset\mathbb R^2$ is compact
and $\dim_{\mathrm H}(E)>1$, then
\[
|E+\Gamma|>0.
\]
More quantitatively, if $I_1(\mu)<\infty$ for a nonzero finite positive
measure $\mu$ supported on $E$, then
\begin{equation}\label{eq:trace-quantitative}
|E+\Gamma|
\geq
c
\frac{\mu(E)^2}{I_1(\mu)}.
\end{equation}
\end{theorem}

\begin{proof}
Since $A\subset\Sigma$, Theorem \ref{thm:curved-graph-overlap} gives
\[
|A^\delta\cap(h+A^\delta)|
\leq
|\Sigma^\delta\cap(h+\Sigma^\delta)|
\leq
C\frac{\delta^2}{\delta+|h|}.
\]
Lemma \ref{lem:trace-tube-volume} gives $|A^\delta|\asymp\delta$. Theorem
\ref{thm:overlap-coverage} therefore gives $|E+A|>0$ whenever
$\dim_{\mathrm H}(E)>1$. If $\mu$ is a nonzero finite positive measure
supported on $E$ with $I_1(\mu)<\infty$, the quantitative part of the same
theorem gives
\[
|E+A|
\geq
c\frac{\mu(E)^2}{I_1(\mu)}.
\]
Finally,
\[
E+A
\subset
E+\Gamma.
\]
This proves both assertions.
\end{proof}

\begin{corollary}\label{cor:curved-trace-threshold}
Let $\Gamma$ be a nondegenerate compact graph having a positive curved trace.
Then
\[
T(\Gamma)=1.
\]
\end{corollary}

\begin{proof}
Let $A\subset\Gamma\cap\Sigma$ be a positive curved trace as in Definition
\ref{def:intro-curved-trace}, and let $S$ be its projection onto the first
coordinate. The parametrization $x\mapsto(x,g(x))$ is bi-Lipschitz on its
compact parameter interval. Hence $S$ is compact, and
\[
\mathcal H^1(A)
\leq
\sqrt{1+\|g'\|_\infty^2}\,|S|.
\]
It follows that $|S|>0$, and
\[
A
=
\{(x,g(x)):x\in S\}.
\]
The upper bound follows from Theorem \ref{thm:positive-curved-trace}. The
lower bound is Proposition \ref{prop:graph-threshold-lower}.
\end{proof}

\begin{corollary}[Quantitative curve projections]
\label{cor:quantitative-curve-projections}
Let $\Gamma$ be a nondegenerate compact rectifiable graph having a positive
curved trace, and let $E\subset\mathbb R^2$ be compact. There are constants
$c>0$ and $0<\delta_0<\frac12$, depending only on the curved trace, such that
for every nonzero finite positive measure $\mu$ supported on $E$ and every
$0<\delta<\delta_0$,
\begin{equation}\label{eq:quantitative-projection-energy}
\operatorname{Fav}_\Gamma(E^\delta)
\geq
c
\frac{\mu(E)^2}{J_\delta(\mu)},
\end{equation}
where $J_\delta(\mu)$ is defined in
\eqref{eq:regularized-one-energy}.

In particular, suppose that $\mu$ is a Borel probability measure supported on
$E$ and satisfying
\begin{equation}\label{eq:frostman-projection-hypothesis}
\mu(B(x,r))
\leq
C_\mu r^s
\end{equation}
for every $x\in\mathbb R^2$ and $0<r\leq1$, where $0<s\leq1$. After decreasing
$c$ and $\delta_0$, if necessary, in a way depending only on $s$ and $C_\mu$,
one has
\begin{equation}\label{eq:quantitative-curve-projections}
\operatorname{Fav}_\Gamma(E^\delta)
\geq
c
\begin{cases}
\delta^{1-s},&0<s<1,\\[4pt]
\displaystyle\left(\log\frac{1}{\delta}\right)^{-1},&s=1,
\end{cases}
\end{equation}
for every $0<\delta<\delta_0$.
\end{corollary}

\begin{proof}
Let $A\subset\Gamma\cap\Sigma$ be a positive curved trace, and let $S$ be its
projection onto the first coordinate. As in the proof of Corollary
\ref{cor:curved-trace-threshold}, one has $|S|>0$ and
\[
A
=
\{(x,g(x)):x\in S\}.
\]
Theorem \ref{thm:curved-graph-overlap} and Lemma
\ref{lem:trace-tube-volume} therefore show that $A$ satisfies the hypotheses
of Theorem \ref{thm:overlap-coverage}. For every compact set
$X\subset\mathbb R^2$, one has
$X^\delta=X+\overline{B(0,\delta)}$. Consequently, both $E^\delta+A$ and
$E+A^\delta$ equal $(E+A)^\delta$. Since $A\subset\Gamma$, one has
\begin{align}
\operatorname{Fav}_\Gamma(E^\delta)
&=
|E^\delta+\Gamma|
\notag\\
&\geq
|E^\delta+A|
=
|E+A^\delta|.
\label{eq:favard-dominates-trace-neighborhood}
\end{align}
Combining \eqref{eq:favard-dominates-trace-neighborhood} with
\eqref{eq:regularized-overlap-coverage} proves
\eqref{eq:quantitative-projection-energy}.

Suppose now that $\mu$ is a probability measure satisfying
\eqref{eq:frostman-projection-hypothesis}. Fix $0<\delta<\frac12$, and choose
an integer $N\geq0$ such that
\[
2^N\delta
\leq
1
<
2^{N+1}\delta.
\]
For each $x\in\mathbb R^2$, split the integral in
\eqref{eq:regularized-one-energy} into the ball $B(x,\delta)$, the annuli
\[
2^j\delta
\leq
|x-y|
<
2^{j+1}\delta,
\qquad 0\leq j<N,
\]
and the remaining region $|x-y|\geq2^N\delta$. Since
$2^N\delta>\frac12$, the last region contributes at most $2$. The Frostman
bound \eqref{eq:frostman-projection-hypothesis} gives
\begin{align}
\int
\frac{d\mu(y)}{\delta+|x-y|}
&\leq
C
+
C C_\mu\delta^{s-1}
\left(
1+
\sum_{j=0}^{N-1}2^{j(s-1)}
\right).
\label{eq:regularized-energy-annuli}
\end{align}
If $0<s<1$, the sum in
\eqref{eq:regularized-energy-annuli} is bounded by a constant depending only
on $s$. Since $\delta^{s-1}\geq1$, integration in $x$ gives
\[
J_\delta(\mu)
\leq
C\delta^{s-1}.
\]
If $s=1$, then $N\leq C\log\frac{1}{\delta}$, and hence
\[
J_\delta(\mu)
\leq
C\log\frac{1}{\delta}.
\]
Combining these estimates with
\eqref{eq:quantitative-projection-energy} proves
\eqref{eq:quantitative-curve-projections}.
\end{proof}

\begin{remark}[Comparison with transversal projection theorems]
\label{rem:projection-comparison}
Under the stronger assumption that $\Gamma$ is piecewise $C^1$ with
piecewise bi-Lipschitz unit tangent, the estimates in Corollary
\ref{cor:quantitative-curve-projections} follow from the transversality theory
of Bongers and Taylor \cite[Theorem~1.7]{BongersTaylor}. The corollary extends
these bounds in a different direction. Only a positive-length trace is
required to lie on a uniformly curved $C^2$ graph; the trace may be nowhere
dense, and no global transversality condition is imposed on the ambient graph.
The result controls the average projection length
\eqref{eq:intro-favard-curve-length}. It does not give an exceptional-parameter
statement of two-projection type. Very recent work of {\L}aba, McDonald, and
Taylor develops a local comparison principle for smooth generalized projection
families and transfers upper bounds from classical Favard length to that setting
\cite{LabaMcDonaldTaylor}. That work also assumes regularity and transversality
of the full family and is complementary to the lower bounds obtained here from
a positive curved trace.
\end{remark}

\begin{remark}\label{rem:trace-hypotheses}
Several features of Corollary \ref{cor:curved-trace-threshold} are useful to
keep in mind. The curved portion of $\Gamma$ need not be an arc, and no
regularity of $\Gamma$ away from the trace is used. In fact, the proof only
needs a compact set
\[
A\subset\Gamma\cap\Sigma
\]
whose projection to the parameter interval has positive Lebesgue measure.
The ambient graph $\Sigma$ supplies the curvature needed for the overlap
estimate, while the positive measure of the projection supplies the lower
bound $|A^\delta|\gtrsim\delta$. Thus the argument is genuinely local in the
curve but measure-theoretic rather than topological: a nowhere-dense trace is
just as effective as an interval.

The positive-length hypothesis also explains the limitation of this
particular method. If the trace had zero length, the normalization in
\eqref{eq:Fdelta} could be much larger than $\delta^{-1}$, and the overlap
bound inherited from $\Sigma$ would no longer reduce the second moment to the
$1$-energy of $\mu$ without additional information on the size of
$A^\delta$. One could try to replace positive length by a quantitative
Minkowski estimate for a thinner trace, but the resulting kernel would change
and the dimension threshold for $E$ would change with it. We do not pursue
that direction here. The purpose of the positive curved trace is precisely to
isolate the geometric hypothesis which yields the sharp threshold one.
\end{remark}

\section{The absolutely continuous part of the curvature measure}
\label{sec:curvature-measure}

Let $\gamma:[a,b]\to\mathbb R$ be convex and Lipschitz. Its right derivative
is nondecreasing. The distributional derivative of this monotone slope is the
finite positive measure $D^2\gamma$. By the Lebesgue decomposition theorem,
we write
\begin{equation}\label{eq:curvature-decomposition}
D^2\gamma
=
k(x)\,dx+\kappa_s,
\end{equation}
where $k(x)\,dx$ is absolutely continuous and $\kappa_s$ is singular with
respect to Lebesgue measure. Thus $k\in L^1(a,b)$ is the density of the
absolutely continuous part, and $\kappa_s$ is the singular part. The density
$k$ is the almost-everywhere derivative of the monotone slope function. To
compare this measure with the intrinsic curvature measure, put
$p=\gamma'_+$ and $\theta=\arctan p$. Since $p$ is bounded, the derivative
of $\arctan$ is bounded above and below by positive constants on the range of
$p$. The mean value theorem therefore gives
\[
c\bigl(p(y)-p(x)\bigr)
\leq
\theta(y)-\theta(x)
\leq
C\bigl(p(y)-p(x)\bigr)
\]
whenever $a<x<y<b$. It follows first on intervals, and then on all Borel sets
by regularity, that $D^2\gamma=Dp$ is comparable to the turning-angle measure
$D\theta$ of the graph. In particular, the two measures have nonzero
absolutely continuous parts simultaneously. The following is the one-variable case of Alexandrov's theorem
\cite{Alexandrov}, whose $n$-dimensional form states that a convex function
has a second-order Taylor expansion almost everywhere. We include the short
proof.

\begin{lemma}[Second-order (Peano) differentiability]
\label{lem:convex-second-order}
For almost every $x_0\in(a,b)$, the function $\gamma$ is differentiable at
$x_0$ and
\begin{equation}\label{eq:convex-second-order}
\gamma(x_0+h)
=
\gamma(x_0)+\gamma'(x_0)h
+
\frac{1}{2}k(x_0)h^2
+
o(h^2)
\end{equation}
as $h\to0$.
\end{lemma}

\begin{proof}
Let $p=\gamma'_+$ be the right derivative of $\gamma$. The function $p$ is
nondecreasing, it agrees with $\gamma'$ almost everywhere, and its
distributional derivative is $D^2\gamma$. By the differentiability theorem
for monotone functions, $p$ is differentiable almost everywhere. Moreover,
at almost every point of differentiability, $p'$ is the Radon--Nikodym
density of the absolutely continuous part of $Dp=D^2\gamma$. Intersecting
this full-measure set with the set on which $\gamma$ is differentiable and
$p=\gamma'$, we obtain, for almost every $x_0$,
\[
p(x_0)=\gamma'(x_0),
\qquad
p'(x_0)=k(x_0).
\]
Fix such a point. Since $\gamma$ is absolutely continuous and $p=\gamma'$
almost everywhere, for $h>0$ one has
\[
\gamma(x_0+h)-\gamma(x_0)
=
\int_0^h p(x_0+t)dt.
\]
The differentiability of $p$ at $x_0$ gives
\[
p(x_0+t)
=
p(x_0)+k(x_0)t+o(|t|).
\]
Integration proves \eqref{eq:convex-second-order} for $h>0$. The argument for
$h<0$ follows from
\[
\gamma(x_0+h)-\gamma(x_0)
=
-\int_h^0 p(x_0+t)dt
\]
and the same expansion of $p(x_0+t)$ as $t\to0$.
\end{proof}

We use the following one-dimensional Lusin approximation theorem.

\begin{theorem}[Goldstein--Haj{\l}asz]\label{thm:goldstein-hajlasz}
Let $f:(a,b)\to\mathbb R$ be convex. For every $\varepsilon>0$ there is a
convex function $g\in C^2(a,b)$ such that
\[
|\{x\in(a,b):f(x)\neq g(x)\}|<\varepsilon.
\]
\end{theorem}

This is Theorem 1.3 of \cite{GoldsteinHajlasz}. Only the stated Lusin
property is needed below.

\begin{theorem}[Positive absolutely continuous curvature]
\label{thm:ac-curvature}
Let $\Gamma$ be the graph of a convex Lipschitz function
$\gamma:[a,b]\to\mathbb R$. Suppose that the density in
\eqref{eq:curvature-decomposition} satisfies
\begin{equation}\label{eq:positive-curvature-density}
\left|\{x\in(a,b):k(x)>0\}\right|>0.
\end{equation}
Then
\[
T(\Gamma)=1.
\]
\end{theorem}

\begin{proof}
Since
\[
\{x\in(a,b):k(x)>0\}
=
\bigcup_{m=1}^\infty\{x\in(a,b):k(x)\geq m^{-1}\},
\]
there is $c_0>0$ such that
\[
S
=
\{x\in(a,b):k(x)\geq c_0\}
\]
has positive measure. Apply Theorem \ref{thm:goldstein-hajlasz} with
$\varepsilon<|S|/2$. There is a convex function $g\in C^2(a,b)$ such that
the equality set
\[
G
=
\{x\in(a,b):\gamma(x)=g(x)\}
\]
satisfies $|G\cap S|>0$.

Choose $x_0\in G\cap S$ which is a Lebesgue density point of $G\cap S$ and
at which the expansion \eqref{eq:convex-second-order} holds. Such points form
a full-measure subset of $G\cap S$. Since $g\in C^2$, the function
$h=\gamma-g$ has an expansion
\begin{equation}\label{eq:h-second-order}
h(x_0+u)
=
h(x_0)+h_1u+\frac{1}{2}h_2u^2+o(u^2),
\end{equation}
where
\[
h_1=\gamma'(x_0)-g'(x_0),
\qquad
h_2=k(x_0)-g''(x_0).
\]
The function $h$ vanishes on $G$, and $h(x_0)=0$. Choose
$x_j\in G\cap S$, $x_j\neq x_0$, with $x_j\to x_0$. Substituting
$u_j=x_j-x_0$ into \eqref{eq:h-second-order} gives
\[
0
=
h_1u_j+\frac{1}{2}h_2u_j^2+o(u_j^2).
\]
Dividing first by $u_j$ and letting $j\to\infty$ gives $h_1=0$.
Substituting this back into the preceding identity, dividing by $u_j^2$, and
letting $j\to\infty$ gives $h_2=0$. Thus
\[
h_1=0,
\qquad
h_2=0.
\]
It follows that
\[
g''(x_0)=k(x_0)\geq c_0.
\]

By continuity, there is a compact interval $I_0\subset(a,b)$ containing
$x_0$ in its interior such that
\[
g''(x)\geq\frac{c_0}{2}
\]
for every $x\in I_0$. Since $x_0$ is a density point of $G\cap S$,
$|G\cap S\cap I_0|>0$. By inner regularity, this set contains a compact
subset $S_0$ of positive measure. Put
\[
A_0
=
\{(x,g(x)):x\in S_0\}.
\]
The set $A_0$ is compact. Since the graphs of $\gamma$ and $g$ coincide
above $S_0$, one has $A_0\subset\Gamma$. The first-coordinate projection from
$A_0$ onto $S_0$ is $1$-Lipschitz and surjective, and hence
$\mathcal H^1(A_0)\geq |S_0|>0$. Thus $\Gamma$ has a positive curved trace.
Corollary
\ref{cor:curved-trace-threshold} proves the theorem.
\end{proof}

\begin{corollary}\label{cor:convex-body-ac-curvature}
Let $K\subset\mathbb R^2$ be a bounded convex body. Suppose that some graph
subarc of $\partial K$ has a curvature measure with a nonzero absolutely
continuous part. Then, for every compact $E\subset\mathbb R^2$,
\[
\dim_{\mathrm H}(E)>1
\quad\Longrightarrow\quad
|E+\partial K|>0.
\]
\end{corollary}

\begin{proof}
Let $\Gamma\subset\partial K$ be the subarc in the hypothesis. After an
orthogonal transformation $Q$, which may include a reflection, and a
translation by $z_0$, we may write
\[
\Gamma'
=
Q\Gamma+z_0
=
\{(x,\gamma(x)):a\leq x\leq b\}
\]
with $\gamma$ convex. These transformations preserve arclength and the
turning-angle measure. Let $\nu$ denote the nonzero absolutely continuous
part of the turning-angle measure on $\Gamma'$. Since $\nu$ is absolutely
continuous with respect to arclength, it gives no mass to the two endpoints.
For all sufficiently large integers $m$, let
\[
\Gamma'_m
=
\{(x,\gamma(x)):a+m^{-1}\leq x\leq b-m^{-1}\}.
\]
These subarcs increase to the graph with its endpoints removed. Hence, for some sufficiently
large $m$, the restriction of $\nu$ to $\Gamma'_m$ is nonzero. Put
\[
[a',b']=[a+m^{-1},b-m^{-1}]
\]
and denote the corresponding subarc by $\Gamma''$.

A finite convex function is Lipschitz on every compact subinterval of the
interior of its domain, so $\gamma$ is Lipschitz on $[a',b']$. On this
interval the graph parametrization
\[
x\longmapsto(x,\gamma(x))
\]
is bi-Lipschitz, and arclength measure is comparable to Lebesgue measure in
$x$. The comparison between $D^2\gamma$ and the turning-angle measure proved
above therefore shows that the absolutely continuous part $k(x)\,dx$ of
$D^2\gamma$ is nonzero on $[a',b']$. Equivalently,
\[
\left|\{x\in[a',b']:k(x)>0\}\right|>0.
\]

Put $E'=QE$. Theorem \ref{thm:ac-curvature}, applied to $\Gamma''$, gives
\[
|E'+\Gamma''|>0.
\]
Since $\Gamma''\subset\Gamma'$ and
\[
E'+\Gamma'
=
Q(E+\Gamma)+z_0,
\]
orthogonal invariance and translation invariance of Lebesgue measure imply
$|E+\Gamma|>0$. Finally, $E+\Gamma\subset E+\partial K$.
\end{proof}

\begin{remark}\label{rem:curvature-regimes}
Theorem \ref{thm:ac-curvature} is deliberately local. A convex curve may
have flat pieces, corners, and singular curvature elsewhere. A positive
amount of absolutely continuous curvature on one measurable part is enough.
Atomic curvature is compatible with polygonal pieces and therefore does not
by itself force the threshold one. For example, every nondegenerate line
segment has threshold two. The upper bound
$T(\Gamma)\leq2$ is automatic, since no subset of $\mathbb R^2$ has
Hausdorff dimension greater than two. To prove the reverse inequality, fix
$t<2$. After applying a rigid motion, which preserves Hausdorff dimension and
planar measure, suppose that
\[
\Gamma=[0,1]\times\{0\}.
\]
Choose Cantor-type self-similar compact null sets $B,C\subset\mathbb R$
satisfying the open set condition and such that
\[
\dim_{\mathrm H}(B)+\dim_{\mathrm H}(C)>t.
\]
For such sets the Hausdorff and packing dimensions agree. More generally, if
$B$ and $C$ are compact and
$\dim_{\mathrm H}(C)=\dim_{\mathrm P}(C)$, then the standard product
inequalities
\[
\dim_{\mathrm H}(B)+\dim_{\mathrm H}(C)
\leq
\dim_{\mathrm H}(B\times C)
\leq
\dim_{\mathrm H}(B)+\dim_{\mathrm P}(C)
\]
show that
$\dim_{\mathrm H}(B\times C)=\dim_{\mathrm H}(B)+\dim_{\mathrm H}(C)$;
see \cite{Mattila}. Thus, for $E=B\times C$, one has
$\dim_{\mathrm H}(E)>t$, whereas
\[
E+\Gamma=(B+[0,1])\times C
\]
has planar measure zero by Fubini's theorem. Since $t<2$ was arbitrary,
$T(\Gamma)\geq2$, and hence $T(\Gamma)=2$. At the same time, Remark
\ref{rem:q-atomic} shows that a purely atomic curvature measure with dense
support may produce a strictly convex primitive rather than a polygonal
piece. Thus atomicity alone does not classify the threshold. Neither the
general dense-atomic case nor the purely singular continuous case is settled
by Theorem \ref{thm:ac-curvature}.
\end{remark}

\section{A convex curve with no Fourier decay on any subarc}
\label{sec:nonrajchman}

Let $F\subset[1,2]$ be a compact nowhere dense set of positive Lebesgue
measure which contains the endpoints $1$ and $2$. Write the complementary
intervals as
\[
[1,2]\setminus F
=
\bigcup_j(a_j,b_j).
\]
Define a nondecreasing function $p_0:[1,2]\to[1,2]$ by
\begin{equation}\label{eq:p0-definition}
p_0(x)
=
\begin{cases}
x,&x\in F,\\[3pt]
\frac{a_j+b_j}{2},&a_j<x<b_j.
\end{cases}
\end{equation}
Set
\begin{equation}\label{eq:gamma0-definition}
\gamma_0(x)
=
\frac12+
\int_1^x p_0(t)dt,
\qquad 1\leq x\leq2,
\end{equation}
and let $\Gamma_0$ be its graph.

\begin{lemma}\label{lem:gamma0-properties}
The function $\gamma_0$ is convex and Lipschitz. Moreover,
\begin{equation}\label{eq:gamma0-parabola}
\gamma_0(x)
=
\frac{x^2}{2}
\end{equation}
for every $x\in F$, and $\gamma_0$ is affine on every complementary interval
$(a_j,b_j)$.
\end{lemma}

\begin{proof}
The function $p_0$ is nondecreasing. Indeed, on a complementary interval
$(a,b)$ its value lies between the endpoint values $p_0(a)=a$ and
$p_0(b)=b$. If $x<y$ do not belong to the same complementary interval, then
the closure of the gap containing $x$, if there is one, lies to the left of
the closure of the gap containing $y$, if there is one; the same endpoint
comparison gives $p_0(x)\leq p_0(y)$. Thus $p_0$ is nondecreasing. Since an
absolutely continuous function whose derivative is nondecreasing almost
everywhere is convex, $\gamma_0$ is convex. Since $1\leq p_0\leq2$, the
function $\gamma_0$ is Lipschitz. On a
complementary interval $(a,b)$,
\begin{align*}
\int_a^b p_0(t)dt
&=
(b-a)\frac{a+b}{2}
=
\int_a^b t\,dt.
\end{align*}
If $x\in F$, the open set $[1,x]\setminus F$ is a countable union of
complete complementary intervals. The displayed calculation shows that the
integral of $p_0-t$ over each such interval is zero, and $p_0(t)=t$ on $F$.
Countable additivity therefore gives
\[
\int_1^x(p_0(t)-t)\,dt=0,
\]
which proves \eqref{eq:gamma0-parabola}. The final assertion follows from
\eqref{eq:p0-definition}.
\end{proof}

The next identity is the elementary form of Wiener's theorem that we need.
We include the proof to keep the Fourier obstruction transparent.

\begin{lemma}[Wiener's mean-square identity]\label{lem:wiener}
Let $\nu$ be a finite positive Borel measure on $\mathbb R$, and use the
Fourier transform convention
\[
\widehat\nu(t)
=
\int e^{-2\pi itx}d\nu(x).
\]
Then
\begin{equation}\label{eq:wiener-identity}
\lim_{R\to\infty}
\frac1{2R}
\int_{-R}^R
|\widehat\nu(t)|^2dt
=
\sum_{x\in\mathbb R}\nu(\{x\})^2.
\end{equation}
\end{lemma}

\begin{proof}
Fubini's theorem gives
\begin{align*}
\frac1{2R}
\int_{-R}^R
|\widehat\nu(t)|^2dt
&=
\iint
\frac{\sin(2\pi R(x-y))}{2\pi R(x-y)}
d\nu(x)d\nu(y),
\end{align*}
where, at $x=y$, the quotient denotes the value of
\[
\frac{1}{2R}\int_{-R}^R e^{-2\pi it(x-y)}\,dt,
\]
namely $1$. The kernel is bounded by $1$ and converges pointwise to the
indicator of the diagonal. Dominated
convergence gives
\[
(\nu\times\nu)(\{(x,x):x\in\mathbb R\})
=
\sum_x\nu(\{x\})^2.
\]
The sum is well defined because a finite measure has at most countably many
atoms.
\end{proof}

\begin{theorem}\label{thm:gamma0-nonrajchman}
The graph $\Gamma_0$ satisfies
\[
T(\Gamma_0)=1.
\]
For every nondegenerate interval $I\subset[1,2]$, let
\[
\Gamma_{0,I}
=
\{(x,\gamma_0(x)):x\in I\}
\]
and
\[
\sigma_I
=
\mathcal H^1|_{\Gamma_{0,I}}.
\]
Then there is a unit vector $\omega_I$ such that
\begin{equation}\label{eq:nonrajchman-conclusion}
\widehat\sigma_I(R\omega_I)
\not\longrightarrow0
\qquad\text{as }R\to\infty.
\end{equation}
\end{theorem}

\begin{proof}
By Lemma \ref{lem:gamma0-properties}, the graph $\Gamma_0$ agrees with the
parabola $y=\frac{x^2}{2}$ over the positive-measure set $F$. Thus it has a
positive curved trace. Corollary \ref{cor:curved-trace-threshold} gives
$T(\Gamma_0)=1$.

Fix a nondegenerate interval $I\subset[1,2]$. Since $F$ is closed and nowhere
dense, $\operatorname{int}(I)\setminus F$ is a nonempty open set. It therefore
contains a nondegenerate interval $J$ lying in a single component of
$[1,2]\setminus F$. By Lemma \ref{lem:gamma0-properties}, $\gamma_0$ is affine
on that component, so
\[
L
=
\{(x,\gamma_0(x)):x\in J\}
\subset
\Gamma_{0,I}
\]
is a nondegenerate line segment and $\mathcal H^1(L)>0$.

Let $\omega_I$ be a unit normal to $L$ and let
\[
\pi_I(z)=z\cdot\omega_I.
\]
Then $\pi_I\equiv c$ on $L$ for some $c\in\mathbb R$. Since $\gamma_0$ is
Lipschitz, $\sigma_I$ is a finite positive Borel measure. Denote its
pushforward under $\pi_I$ by
\[
\nu_I=(\pi_I)_\#\sigma_I.
\]
This measure has an atom at $c$, because
\[
\nu_I(\{c\})
=
\sigma_I\bigl(\pi_I^{-1}(\{c\})\bigr)
\geq
\sigma_I(L)
=
\mathcal H^1(L)
>
0.
\]
By the definition of the pushforward,
\begin{align*}
\widehat\sigma_I(R\omega_I)
&=
\int e^{-2\pi iRz\cdot\omega_I}\,d\sigma_I(z)
\\
&=
\int e^{-2\pi iRt}\,d\nu_I(t)
=
\widehat\nu_I(R).
\end{align*}
Lemma \ref{lem:wiener} therefore gives
\begin{equation}\label{eq:positive-mean-square}
\lim_{R\to\infty}
\frac{1}{2R}\int_{-R}^R|\widehat\nu_I(t)|^2\,dt
=
\sum_x\nu_I(\{x\})^2
\geq
\mathcal H^1(L)^2
>
0.
\end{equation}

Suppose that $\widehat\nu_I(R)\to0$ as $R\to+\infty$. Since $\nu_I$ is a
positive measure,
\[
\widehat\nu_I(-t)=\overline{\widehat\nu_I(t)},
\]
so $|\widehat\nu_I(t)|\to0$ as $|t|\to\infty$. Given $\varepsilon>0$, choose
$T>0$ such that $|\widehat\nu_I(t)|\leq\varepsilon$ for $|t|\geq T$. Since
$|\widehat\nu_I(t)|\leq\nu_I(\mathbb R)$ for all $t$, one has, for $R\geq T$,
\[
\frac{1}{2R}\int_{-R}^R|\widehat\nu_I(t)|^2\,dt
\leq
\frac{T\nu_I(\mathbb R)^2}{R}
+
\varepsilon^2.
\]
Letting $R\to\infty$ and then $\varepsilon\downarrow0$ contradicts
\eqref{eq:positive-mean-square}. This proves
\eqref{eq:nonrajchman-conclusion}.
\end{proof}

\begin{remark}
The conclusion is stronger than the absence of a polynomial decay estimate.
Arclength on every nontrivial subarc fails even the Rajchman property. The
positive-measure theorem is nevertheless optimal, since
$T(\Gamma_0)=1$.
\end{remark}

\section{A strictly convex curve with no polynomial decay}
\label{sec:strict-example}

We now remove the line segments from the preceding construction. The first
step is a strictly increasing function whose oscillation on dyadic intervals
is much smaller than the separation between neighboring slope ranges.

Put
\begin{equation}\label{eq:weight-sum}
S
=
\sum_{k=1}^\infty2^{-k^2},
\qquad
c_k
=
S^{-1}2^{-k^2},
\qquad
T_n
=
\sum_{k>n}c_k.
\end{equation}

\begin{lemma}\label{lem:superincreasing-weights}
For every $n\geq1$,
\begin{equation}\label{eq:tail-estimates}
T_n
\leq
2c_{n+1},
\qquad
c_n-T_n
\geq
\frac12c_n,
\qquad
\frac{T_n}{c_n}
\leq
2^{-2n}.
\end{equation}
\end{lemma}

\begin{proof}
For $k\geq n+1$,
\[
\frac{c_{k+1}}{c_k}
=
2^{-(2k+1)}
\leq
2^{-(2n+3)}.
\]
Since $2^{-(2n+3)}\leq\frac{1}{32}$, summing the resulting geometric series
gives
\[
T_n
\leq
\frac{c_{n+1}}{1-2^{-(2n+3)}}
\leq
2c_{n+1}.
\]
As $c_{n+1}/c_n=2^{-(2n+1)}$, it follows that
\[
\frac{T_n}{c_n}
\leq
2\frac{c_{n+1}}{c_n}
=
2^{-2n}.
\]
In particular, $T_n/c_n\leq\frac14$ for $n\geq1$, and hence
\[
c_n-T_n
\geq
\frac34c_n
\geq
\frac12c_n.
\]
\end{proof}

For $t\in[0,1)$, write
\[
t
=
\sum_{k=1}^\infty\omega_k(t)2^{-k},
\qquad
\omega_k(t)\in\{0,1\},
\]
using the binary expansion which is not eventually equal to $1$. Define
\begin{equation}\label{eq:q-definition}
q(t)
=
\sum_{k=1}^\infty c_k\omega_k(t),
\qquad 0\leq t<1,
\end{equation}
and set $q(1)=1$.

\begin{lemma}\label{lem:q-properties}
The function $q:[0,1]\to[0,1]$ is strictly increasing and
\begin{equation}\label{eq:q-mean}
\int_0^1 q(t)dt
=
\frac12.
\end{equation}
Let $D$ be a dyadic interval of level $n$, interpreted as half-open except at
the right endpoint of $[0,1]$. Then, up to endpoints,
\begin{equation}\label{eq:q-cylinder-oscillation}
\operatorname{diam}q(D)
\leq
T_n
\end{equation}
and, except for the dyadic endpoints,
\begin{equation}\label{eq:q-cylinder-separation}
|q(s)-q(t)|
\geq
\frac{1}{2}c_n
\end{equation}
whenever $s\in D$ and $t\in[0,1]\setminus D$.
\end{lemma}

\begin{proof}
Suppose first that $0\leq s<t<1$, and let $j$ be the first binary digit at
which their chosen expansions differ. Then $\omega_j(s)=0$ and
$\omega_j(t)=1$. Hence
\[
q(t)-q(s)
\geq
c_j-
\sum_{k>j}c_k
=
c_j-T_j
>
0
\]
by Lemma \ref{lem:superincreasing-weights}. If $s<1$, its chosen binary
expansion has at least one zero digit, and therefore $q(s)<\sum_k c_k=1=q(1)$.
Thus $q$ is strictly increasing on $[0,1]$. For each $k$, the set
\[
\{t\in[0,1):\omega_k(t)=1\}
\]
is the union of the $2^{k-1}$ dyadic intervals of level $k$ whose $k$-th
digit equals $1$. Each interval has length $2^{-k}$, so this set has measure
$\frac12$ and
\[
\int_0^1\omega_k(t)\,dt=\frac12.
\]
Tonelli's theorem now gives
\[
\int_0^1 q(t)dt
=
\frac12
\sum_{k=1}^\infty c_k
=
\frac12.
\]
On a level-$n$ dyadic interval, the first $n$ digits are fixed, so only the
tail contributes to the oscillation. This proves
\eqref{eq:q-cylinder-oscillation}. Suppose next that $s\in D$ and
$t\notin D$, and exclude the dyadic endpoints so that the first $n$ digits
identify the cylinder without ambiguity. Their binary expansions first
differ at some position $j\leq n$. The contribution of this digit has size
$c_j$, while the total contribution of all later digits is at most $T_j$.
Therefore
\[
|q(s)-q(t)|
\geq
c_j-T_j
\geq
\frac{1}{2}c_j
\geq
\frac{1}{2}c_n,
\]
which is \eqref{eq:q-cylinder-separation}. The excluded endpoints form a
null set and do not affect any later integral.
\end{proof}

\begin{remark}[Atomic curvature of the Cantor-quantile pieces]
\label{rem:q-atomic}
The function $q$ is strictly increasing, but its distributional derivative is
purely atomic. To see this, let $j2^{-n}\in(0,1)$ be a dyadic rational of exact
level $n$, so that $j$ is odd. With the binary convention used above, the jump
of $q$ at this point is
\[
q(j2^{-n})-q(j2^{-n}-)
=
c_n-T_n.
\]
There are $2^{n-1}$ dyadic rationals of exact level $n$. The total mass of all
these jumps is
\begin{align*}
\sum_{n=1}^\infty 2^{n-1}(c_n-T_n)
&=
\sum_{n=1}^\infty 2^{n-1}c_n
-
\sum_{n=1}^\infty 2^{n-1}\sum_{k>n}c_k
\\
&=
\sum_{k=1}^\infty
\left(2^{k-1}-\sum_{n=1}^{k-1}2^{n-1}\right)c_k
\\
&=
\sum_{k=1}^\infty c_k
=
1.
\end{align*}
This is the full increase $q(1)-q(0)$. Since the distributional derivative of
a monotone function is the sum of its jump atoms and a non-atomic positive
measure, it follows that
\[
Dq
=
\sum_{n=1}^\infty
\sum_{\substack{1\leq j<2^n\\ j\ \mathrm{odd}}}
(c_n-T_n)\,\delta_{j2^{-n}}.
\]
Thus $Dq$ is purely atomic and its atoms are dense in $[0,1]$. Nevertheless,
$q$ is strictly increasing, and hence its primitive is strictly convex. After
affine rescaling, the same description holds for the slope on every
complementary interval in \eqref{eq:p1-definition}. The full slope $p_1$ also
satisfies $p_1(x)=x$ on the positive-measure set $F$. At almost every density
point of $F$ at which the monotone function $p_1$ is differentiable, difference
quotients along $F$ show that $p_1'(x)=1$. Thus the curvature measure of the
full graph $\Gamma_1$ has an additional nonzero absolutely continuous part.
The point of the present remark is that the Cantor-quantile pieces show that
purely atomic curvature need not lead to polygonal behavior.
\end{remark}

Let $F\subset[1,2]$ be the fat Cantor set used in Section
\ref{sec:nonrajchman}. For every complementary interval $(a,b)$, put
$\ell=b-a$ and define
\begin{equation}\label{eq:p1-definition}
p_1(x)
=
\begin{cases}
x,&x\in F,\\[3pt]
a+\ell q\left(\frac{x-a}{\ell}\right),&a<x<b.
\end{cases}
\end{equation}
Set
\begin{equation}\label{eq:gamma1-definition}
\gamma_1(x)
=
\frac12+
\int_1^x p_1(t)dt,
\qquad 1\leq x\leq2,
\end{equation}
and let $\Gamma_1$ be its graph.

Figure \ref{fig:slope-functions} summarizes the relation among the slope
functions used in the two constructions. The drawings are schematic: the
gaps are enlarged, and the middle panel shows a finite-stage approximation to
$q$ so that its jump structure is visible.

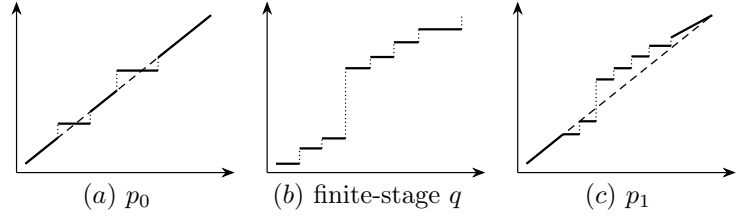
\begin{figure}[H]
\centering
\begin{tikzpicture}[x=.78cm,y=.78cm,>=Stealth]
  \begin{scope}[shift={(-5.2,0)}]
    \draw[->,line width=.45pt] (-.15,-.15)--(3.55,-.15);
    \draw[->,line width=.45pt] (-.15,-.15)--(-.15,2.75);
    \draw[densely dashed,line width=.5pt] (0,0)--(3.15,2.52);
    \draw[line width=.85pt] (0,0)--(.55,.44);
    \draw[line width=.85pt] (.55,.68)--(1.10,.68);
    \draw[densely dotted,line width=.45pt] (.55,.44)--(.55,.68);
    \draw[densely dotted,line width=.45pt] (1.10,.68)--(1.10,.88);
    \draw[line width=.85pt] (1.10,.88)--(1.55,1.24);
    \draw[line width=.85pt] (1.55,1.58)--(2.25,1.58);
    \draw[densely dotted,line width=.45pt] (1.55,1.24)--(1.55,1.58);
    \draw[densely dotted,line width=.45pt] (2.25,1.58)--(2.25,1.80);
    \draw[line width=.85pt] (2.25,1.80)--(3.15,2.52);
    \node[font=\small] at (1.55,-.55) {$(a)$ $p_0$};
  \end{scope}

  \begin{scope}[shift={(-.95,0)}]
    \draw[->,line width=.45pt] (-.15,-.15)--(3.55,-.15);
    \draw[->,line width=.45pt] (-.15,-.15)--(-.15,2.75);
    \draw[line width=.85pt] (0,0)--(.40,0);
    \draw[densely dotted,line width=.45pt] (.40,0)--(.40,.26);
    \draw[line width=.85pt] (.40,.26)--(.78,.26);
    \draw[densely dotted,line width=.45pt] (.78,.26)--(.78,.43);
    \draw[line width=.85pt] (.78,.43)--(1.18,.43);
    \draw[densely dotted,line width=.45pt] (1.18,.43)--(1.18,1.62);
    \draw[line width=.85pt] (1.18,1.62)--(1.60,1.62);
    \draw[densely dotted,line width=.45pt] (1.60,1.62)--(1.60,1.81);
    \draw[line width=.85pt] (1.60,1.81)--(2.00,1.81);
    \draw[densely dotted,line width=.45pt] (2.00,1.81)--(2.00,2.06);
    \draw[line width=.85pt] (2.00,2.06)--(2.42,2.06);
    \draw[densely dotted,line width=.45pt] (2.42,2.06)--(2.42,2.28);
    \draw[line width=.85pt] (2.42,2.28)--(3.15,2.28);
    \draw[densely dotted,line width=.45pt] (3.15,2.28)--(3.15,2.52);
    \node[font=\small] at (1.55,-.55) {$(b)$ finite-stage $q$};
  \end{scope}

  \begin{scope}[shift={(3.3,0)}]
    \draw[->,line width=.45pt] (-.15,-.15)--(3.55,-.15);
    \draw[->,line width=.45pt] (-.15,-.15)--(-.15,2.75);
    \draw[densely dashed,line width=.5pt] (0,0)--(3.15,2.52);
    \draw[line width=.85pt] (0,0)--(.62,.50);
    \draw[line width=.85pt] (.62,.50)--(.90,.50);
    \draw[densely dotted,line width=.45pt] (.90,.50)--(.90,.72);
    \draw[line width=.85pt] (.90,.72)--(1.18,.72);
    \draw[densely dotted,line width=.45pt] (1.18,.72)--(1.18,1.43);
    \draw[line width=.85pt] (1.18,1.43)--(1.48,1.43);
    \draw[densely dotted,line width=.45pt] (1.48,1.43)--(1.48,1.62);
    \draw[line width=.85pt] (1.48,1.62)--(1.78,1.62);
    \draw[densely dotted,line width=.45pt] (1.78,1.62)--(1.78,1.82);
    \draw[line width=.85pt] (1.78,1.82)--(2.08,1.82);
    \draw[densely dotted,line width=.45pt] (2.08,1.82)--(2.08,2.00);
    \draw[line width=.85pt] (2.08,2.00)--(2.45,2.00);
    \draw[densely dotted,line width=.45pt] (2.45,2.00)--(2.45,2.14);
    \draw[line width=.85pt] (2.45,2.14)--(3.15,2.52);
    \node[font=\small] at (1.55,-.55) {$(c)$ $p_1$};
  \end{scope}
\end{tikzpicture}
\caption{Schematic slope functions. The dashed diagonal represents the
identity function. The function $p_0$ agrees with the identity on $F$ and is
constant on each complementary interval. The auxiliary function $q$ is a
strictly increasing pure-jump function; panel $(b)$ displays only finitely
many of its dyadic jumps. The function $p_1$ agrees with the identity on $F$
and inserts an affine copy of $q$ in every gap. Panels $(b)$ and $(c)$ use
finite-stage depictions to make the jumps visible; in the actual functions the
dyadic jumps are dense, so no horizontal interval remains in $q$ or in the
inserted pieces of $p_1$. The dotted vertical connectors indicate jumps and
are not part of the graphs.}
\label{fig:slope-functions}
\end{figure}

Figure \ref{fig:strict-multiscale-piece} shows the two features of the
construction that will be used below. A small dyadic interval $J_n$ inside a
gap has a very small slope oscillation, while its slope range is separated
from the slopes outside $J_n$ by a much larger quantity.

\begin{figure}[H]
\centering
\begin{tikzpicture}[x=1cm,y=1cm,>=Stealth]
  \begin{scope}[shift={(-3.75,0)}]
    \draw[->,line width=.45pt] (-.35,-1.05)--(4.85,-1.05);
    \draw[->,line width=.45pt] (-.35,-1.05)--(-.35,1.72);
    \draw[line width=.85pt]
      plot[samples=120,domain=0:4.5] (\x,{0.075*\x*\x-.82});
    \draw[line width=2.5pt]
      plot[samples=30,domain=2.05:2.92] (\x,{0.075*\x*\x-.82});
    \draw[densely dashed,line width=.55pt]
      (1.45,-.40)--(3.65,.59);
    \draw[->,line width=.65pt] (2.50,-.31)--++(-.46,.98)
      node[above left,font=\small] {$\omega_n$};
    \draw[<->,line width=.5pt] (2.05,-.84)--(2.92,-.84)
      node[midway,below=2pt,font=\scriptsize] {$J_n$};
    \node[font=\scriptsize,align=center] at (2.30,1.38)
      {the phase is almost constant\\on the highlighted piece};
  \end{scope}

  \begin{scope}[shift={(3.05,0)}]
    \draw[->,line width=.45pt] (-.10,-.05)--(4.75,-.05)
      node[right,font=\scriptsize] {slope};
    \draw[line width=3.0pt] (.15,-.05)--(1.10,-.05);
    \draw[line width=4.2pt] (2.20,-.05)--(2.48,-.05);
    \draw[line width=3.0pt] (3.55,-.05)--(4.50,-.05);
    \draw (2.34,-.13)--(2.34,.13);
    \node[font=\scriptsize] at (.62,.30) {left slopes};
    \node[font=\scriptsize] at (4.02,.30) {right slopes};
    \node[font=\scriptsize] at (2.34,.28) {$m_n$};
    \node[font=\small] at (2.34,-.38) {$K_n$};
    \draw[<->,line width=.5pt] (2.20,.66)--(2.48,.66)
      node[midway,above=2pt,font=\scriptsize] {$|K_n|\leq\ell T_n$};
    \draw[<->,line width=.5pt] (1.13,-.64)--(2.17,-.64)
      node[midway,below=2pt,font=\scriptsize] {$\geq\frac12\ell c_n$};
    \draw[<->,line width=.5pt] (2.51,-.64)--(3.52,-.64)
      node[midway,below=2pt,font=\scriptsize] {$\geq\frac12\ell c_n$};
  \end{scope}
\end{tikzpicture}
\caption{The two scales used in the Fourier argument. On the level-$n$
interval $J_n$, the essential slope range is contained in an interval $K_n$
of length at most $\ell T_n$. Every slope from the same complementary gap but
outside $J_n$ lies at distance at least $\frac12\ell c_n$ from $K_n$; slopes
outside the gap obey a comparable bound for all large $n$. Since
$T_n/c_n\to0$ rapidly, a frequency in the normal direction $\omega_n$ is
nearly stationary on $J_n$ and uniformly nonstationary on its complement.}
\label{fig:strict-multiscale-piece}
\end{figure}
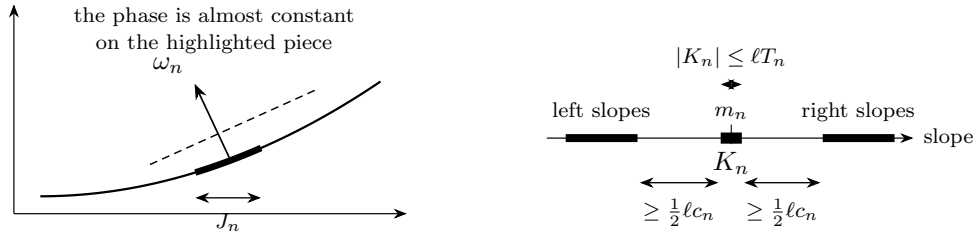

\begin{theorem}\label{thm:gamma1-geometric-properties}
The function $\gamma_1$ is Lipschitz and strictly convex. It satisfies
\begin{equation}\label{eq:gamma1-parabola}
\gamma_1(x)
=
\frac{x^2}{2}
\end{equation}
for every $x\in F$. Consequently,
\[
T(\Gamma_1)=1.
\]
\end{theorem}

\begin{proof}
The function $p_1$ is bounded between $1$ and $2$. We first check that it is
strictly increasing. It is strictly increasing inside each complementary
interval because $q$ is strictly increasing. On a gap $(a,b)$ its values lie
strictly between $a$ and $b$, while $p_1(a)=a$ and $p_1(b)=b$. If two points
belong to different gaps, or if one of them belongs to $F$, the order of the
corresponding values of $p_1$ follows from the order of the closures of the
gaps. Thus
\[
x<y
\quad\Longrightarrow\quad
p_1(x)<p_1(y).
\]

If $x_1<x_2<x_3$, every value of $p_1$ on $(x_1,x_2)$ is smaller than every
value on $(x_2,x_3)$. Hence the average of $p_1$ on the first interval is
strictly smaller than its average on the second interval. More explicitly,
strict monotonicity gives
\[
p_1(t)<p_1(x_2)
\quad\text{for }x_1<t<x_2,
\qquad
p_1(t)>p_1(x_2)
\quad\text{for }x_2<t<x_3.
\]
The two differences are positive measurable functions on intervals of
positive length, so
\[
\frac{1}{x_2-x_1}\int_{x_1}^{x_2}p_1(t)dt
<
p_1(x_2)
<
\frac{1}{x_3-x_2}\int_{x_2}^{x_3}p_1(t)dt.
\]
These two averages are the successive secant slopes of $\gamma_1$. The
strict increase of successive secant slopes is the strict convexity of
$\gamma_1$. The boundedness of $p_1$ gives the Lipschitz property.

On a complementary interval $(a,b)$, the change of variables
$t=a+\ell u$ and \eqref{eq:q-mean} give
\begin{align*}
\int_a^b p_1(t)dt
&=
a\ell+\ell^2\int_0^1 q(u)du
\notag\\
&=
a\ell+\frac{\ell^2}{2}
=
\int_a^b t\,dt.
\end{align*}
The set $[1,x]\setminus F$ is a countable union of complete complementary
intervals whenever $x\in F$. On each of these intervals the preceding
calculation shows that $p_1$ and the identity function have the same
integral, while $p_1(t)=t$ on $F$. Therefore
\eqref{eq:gamma1-parabola} holds for every $x\in F$. Hence $\Gamma_1$ has a
positive curved trace on the parabola. Corollary
\ref{cor:curved-trace-threshold} gives $T(\Gamma_1)=1$.
\end{proof}

The construction has been chosen so that two scales are separated very
strongly. On a level-$n$ dyadic interval the slope oscillates by at most
$\ell T_n$, so a frequency of size roughly
$(\ell^2 2^{-n}T_n)^{-1}$ sees that piece as almost stationary. At the same
time, slopes away from that interval are separated by order $\ell c_n$.
Because $T_n/c_n$ decays exponentially in $n$, the contribution of the
complement can be estimated by a first-derivative bound and is negligible
compared with the length of the almost-stationary piece. The
superexponential choice $c_n\asymp2^{-n^2}$ then makes the selected
frequencies grow so rapidly that the resulting Fourier lower bounds defeat
every fixed power of $|\xi|$.

We now prove the Fourier statement. We first record a standard
first-derivative estimate in the form suited to monotone slopes.

\begin{lemma}[First derivative estimate]\label{lem:first-derivative}
Let $I=[a,b]$, let $\phi:I\to\mathbb R$ be absolutely continuous, and suppose
that $\phi'$ agrees almost everywhere with a monotone function. Let
$A:I\to\mathbb C$ have bounded variation. Suppose that $\phi'$ has constant
sign almost everywhere and
\[
|\phi'(x)|\geq\lambda>0
\]
for almost every $x\in I$. Then, for every $R>0$,
\begin{equation}\label{eq:first-derivative}
\left|
\int_I e^{-2\pi iR\phi(x)}A(x)dx
\right|
\leq
\frac{C}{R\lambda}
\left(
\|A\|_\infty+\operatorname{Var}_I(A)
\right).
\end{equation}
\end{lemma}

\begin{proof}
The argument below applies verbatim with either sign in the exponential. After
replacing $\phi$ by $-\phi$ if necessary, we may therefore suppose that
$\phi'(x)\geq\lambda$ almost everywhere. Let $m$ be a monotone representative
of $\phi'$. The set on which $m<\lambda$ has measure zero. Replacing $m$ by
$\max\{m,\lambda\}$, which preserves monotonicity and changes $m$ only on
that null set, we may assume that $m(x)\geq\lambda$ for every $x\in I$.

For $x<y$, absolute continuity gives
\[
\phi(y)-\phi(x)
=
\int_x^y m(t)\,dt
\geq
\lambda(y-x).
\]
Thus $\phi$ is strictly increasing, and its inverse
\[
\psi=\phi^{-1}:\phi(I)\longrightarrow I
\]
is Lipschitz with constant at most $\lambda^{-1}$. In particular, $\psi$ is
absolutely continuous.

Fix a subinterval $J=[u,v]\subset I$. The change of variables $w=\phi(x)$
gives
\begin{equation}\label{eq:first-derivative-change-variable}
\int_J e^{-2\pi iR\phi(x)}\,dx
=
\int_{\phi(u)}^{\phi(v)}e^{-2\pi iRw}\psi'(w)\,dw.
\end{equation}
The exceptional set on which $\phi'$ either does not exist or differs from
$m$ has measure zero. Since an absolutely continuous function has Luzin's
property $(N)$, the image of this exceptional set under $\phi$ also has
measure zero. The usual
one-dimensional inverse-derivative formula therefore gives
\[
\psi'(w)
=
\frac{1}{m(\psi(w))}
\]
for almost every $w\in\phi(I)$. Define the pointwise representative
\[
\beta(w)
=
\frac{1}{m(\psi(w))},
\qquad w\in\phi(I).
\]
If $m$ is nondecreasing, then $\beta$ is nonincreasing; if $m$ is
nonincreasing, then $\beta$ is nondecreasing. In either case,
\[
0\leq\beta(w)\leq\lambda^{-1},
\qquad
\operatorname{Var}_{\phi(J)}(\beta)\leq\lambda^{-1}.
\]
We replace the almost-everywhere derivative $\psi'$ in
\eqref{eq:first-derivative-change-variable} by this monotone representative
before using Riemann--Stieltjes integration by parts.

Set
\[
H(y)
=
\int_{\phi(u)}^y e^{-2\pi iRw}\,dw.
\]
Then $H(\phi(u))=0$ and $\|H\|_\infty\leq(\pi R)^{-1}$. Hence
\begin{align*}
\left|
\int_J e^{-2\pi iR\phi(x)}\,dx
\right|
&=
\left|
\int_{\phi(u)}^{\phi(v)}\beta(w)\,dH(w)
\right|
\\
&\leq
\|H\|_\infty
\left(
\|\beta\|_\infty+
\operatorname{Var}_{\phi(J)}(\beta)
\right)
\\
&\leq
\frac{2}{\pi R\lambda}.
\end{align*}
This estimate is uniform over all subintervals $J\subset I$.

Now put
\[
P(x)
=
\int_a^x e^{-2\pi iR\phi(t)}\,dt.
\]
The preceding estimate, applied to $[a,x]$, gives
\[
\|P\|_\infty\leq\frac{2}{\pi R\lambda}.
\]
Moreover, $P(a)=0$ and
$P'(x)=e^{-2\pi iR\phi(x)}$ almost everywhere. Therefore
\begin{align*}
\int_I e^{-2\pi iR\phi(x)}A(x)\,dx
&=
\int_I A(x)\,dP(x)
\\
&=
A(b)P(b)-A(a)P(a)-\int_I P(x)\,dA(x)
\\
&=
A(b)P(b)-\int_I P(x)\,dA(x).
\end{align*}
Taking absolute values yields
\[
\left|
\int_I e^{-2\pi iR\phi(x)}A(x)\,dx
\right|
\leq
\frac{2}{\pi R\lambda}
\left(
\|A\|_\infty+
\operatorname{Var}_I(A)
\right),
\]
which proves \eqref{eq:first-derivative}.
\end{proof}

Recall that our Fourier transform convention is
\begin{equation}\label{eq:fourier-convention}
\widehat\sigma(\xi)
=
\int e^{-2\pi ix\cdot\xi}d\sigma(x).
\end{equation}

\begin{theorem}\label{thm:gamma1-no-polynomial-decay}
Let $I\subset[1,2]$ be a nondegenerate compact interval, let
\[
\Gamma_{1,I}
=
\{(x,\gamma_1(x)):x\in I\},
\qquad
\sigma_I
=
\mathcal H^1|_{\Gamma_{1,I}}.
\]
For every $\alpha>0$,
\begin{equation}\label{eq:no-polynomial-decay}
\limsup_{|\xi|\to\infty}
|\xi|^\alpha|\widehat\sigma_I(\xi)|
=
\infty.
\end{equation}
\end{theorem}

\begin{proof}
Since $F$ is nowhere dense, there are a complementary interval $(a,b)$ of
$F$ and a nonempty open interval $U$ such that
\[
U\subset\operatorname{int}(I)\cap(a,b).
\]
Put $\ell=b-a$. Choose a non-dyadic point $t_0\in(0,1)$ such that
$a+\ell t_0\in U$. For all sufficiently large $n_0$, the closure of the
level-$n_0$ dyadic interval containing $t_0$ is carried into $U$ by the
affine map $t\mapsto a+\ell t$. For every $n\geq n_0$, let $D_n$ be the
level-$n$ half-open dyadic interval containing $t_0$ and put
\[
J_n
=
a+\ell D_n.
\]
The endpoints of $J_n$ play no role in the integrals below. We have
$J_n\subset I$ and
\begin{equation}\label{eq:Jn-length}
|J_n|
=
L_n
=
\ell2^{-n}.
\end{equation}

Let $K_n$ be the smallest closed interval containing the essential range of
$p_1$ on $J_n$. By Lemma \ref{lem:q-properties}, its length is at most
\begin{equation}\label{eq:slope-oscillation}
\Delta_n
=
\ell T_n.
\end{equation}
Let $m_n$ be the midpoint of $K_n$ and put
\begin{equation}\label{eq:omega-n}
\omega_n
=
\frac{(-m_n,1)}{\sqrt{1+m_n^2}}.
\end{equation}
For
\[
\Phi_n(x)
=
\omega_n\cdot(x,\gamma_1(x)),
\]
one has, at almost every differentiability point of $\gamma_1$,
\begin{equation}\label{eq:phase-derivative-formula}
\Phi_n'(x)
=
\frac{p_1(x)-m_n}{\sqrt{1+m_n^2}}.
\end{equation}
For almost every $x\in J_n$, the value $p_1(x)$ lies in $K_n$. Since $m_n$
is the midpoint of $K_n$ and $|K_n|\leq\Delta_n$, one has
\[
|p_1(x)-m_n|
\leq
\frac{\Delta_n}{2}.
\]
Since $\sqrt{1+m_n^2}\geq1$, it follows that
\begin{equation}\label{eq:phase-derivative-on-Jn}
|\Phi_n'(x)|
\leq
\frac{\Delta_n}{2}
\end{equation}
for almost every $x\in J_n$. Therefore
\begin{equation}\label{eq:phase-variation}
\operatorname{osc}_{J_n}\Phi_n
\leq
\frac{1}{2}L_n\Delta_n.
\end{equation}

Since $\gamma_1'=p_1$ almost everywhere, the graph parametrization with
arclength density gives
\begin{equation}\label{eq:strict-fourier-parametrization}
\widehat\sigma_I(R\omega_n)
=
\int_I e^{-2\pi iR\Phi_n(x)}
\sqrt{1+p_1(x)^2}\,dx
\end{equation}
for every $R>0$. Choose a fixed $\varepsilon_0$ with
$0<\varepsilon_0<\frac13$, and set
\begin{equation}\label{eq:Rn-definition}
R_n
=
\frac{\varepsilon_0}{L_n\Delta_n}.
\end{equation}
By \eqref{eq:phase-variation}, the arguments of the complex numbers
$e^{-2\pi iR_n\Phi_n(x)}$, $x\in J_n$, lie in an interval of length less
than $\frac{\pi}{3}$. After multiplication by a complex number of modulus
one, they therefore all have real part at least
$\cos(\frac{\pi}{6})$. Since the arclength density
\[
A(x)
=
\sqrt{1+p_1(x)^2}
\]
is at least $1$, we obtain
\begin{equation}\label{eq:local-fourier-lower}
\left|
\int_{J_n}
e^{-2\pi iR_n\Phi_n(x)}A(x)dx
\right|
\geq
cL_n.
\end{equation}

We next estimate the contribution from $I\setminus J_n$. If $x$ belongs to the same complementary interval $(a,b)$ but not to
$J_n$, then \eqref{eq:q-cylinder-separation} gives, outside a null set, a
separation of at least $\frac{1}{2}\ell c_n$ from every slope value on
$J_n$. Since $q$ is increasing, a point outside $D_n$ lies entirely to one
side of the slope interval $K_n$. Hence
\[
\operatorname{dist}(p_1(x),K_n)
\geq
\frac{1}{2}\ell c_n.
\]
Since $m_n\in K_n$, this gives
\begin{equation}\label{eq:same-gap-slope-separation}
|p_1(x)-m_n|
\geq
\frac{1}{2}\ell c_n.
\end{equation}

It remains to compare with slopes outside $(a,b)$. Every point of $D_n$
has the same first $n$ binary digits as $t_0$, apart from the irrelevant
dyadic endpoint. Hence every point in the essential range of $q$ on $D_n$
differs from $q(t_0)$ by at most $T_n$. Since $T_n\to0$, the interval $K_n$
shrinks to the single value $a+\ell q(t_0)$. In particular,
$m_n\to a+\ell q(t_0)$. Since $t_0\in(0,1)$ and $q$ is strictly increasing,
\[
d_0
=
\min\{q(t_0),1-q(t_0)\}
>
0.
\]
For all sufficiently large $n$, the distance from $m_n$ to both endpoint
slopes $a$ and $b$ is at least $\frac{1}{2}\ell d_0$. Monotonicity of
$p_1$ then gives the same lower bound for every point outside $(a,b)$.
Since $c_n\leq c_1$, this lower bound is at least a constant multiple of
$\ell c_n$. Together with \eqref{eq:same-gap-slope-separation}, this proves
\begin{equation}\label{eq:global-slope-separation}
|p_1(x)-m_n|
\geq
c\ell c_n
\end{equation}
for almost every $x\in I\setminus J_n$.

Because $p_1$ is strictly increasing and $m_n$ lies between the endpoints of
$K_n$, the derivative in \eqref{eq:phase-derivative-formula} is negative
on the component of $I\setminus J_n$ to the left of $J_n$ and positive on
the component to its right. On either component,
\[
|\Phi_n'|
\geq
c\ell c_n.
\]
The amplitude $A$ is monotone because $1\leq p_1\leq2$ and the function
$p\mapsto\sqrt{1+p^2}$ is increasing. In particular,
\[
\|A\|_\infty\leq\sqrt5,
\qquad
\operatorname{Var}_I(A)\leq\sqrt5-\sqrt2,
\]
uniformly in $n$. Applying Lemma \ref{lem:first-derivative} separately to the
at most two interval components of $I\setminus J_n$, after adjoining their
endpoints, which does not change the integrals, gives
\begin{align}
\left|
\int_{I\setminus J_n}
e^{-2\pi iR_n\Phi_n(x)}A(x)dx
\right|
&\leq
\frac{C}{R_n\ell c_n}
\notag\\
&=
C L_n\frac{T_n}{c_n}
\notag\\
&\leq
C L_n2^{-2n}.
\label{eq:complement-fourier-upper}
\end{align}
Here we used \eqref{eq:Rn-definition},
\eqref{eq:slope-oscillation}, and Lemma
\ref{lem:superincreasing-weights}.

The integral in \eqref{eq:strict-fourier-parametrization} is the sum of its
parts over $J_n$ and $I\setminus J_n$. The reverse triangle inequality,
\eqref{eq:local-fourier-lower}, and
\eqref{eq:complement-fourier-upper} therefore give
\begin{equation}\label{eq:fourier-lower-final}
|\widehat\sigma_I(R_n\omega_n)|
\geq
cL_n
\end{equation}
for all sufficiently large $n$.

Finally,
\[
c_{n+1}
\leq
T_n
\leq
2c_{n+1},
\]
so $T_n\asymp2^{-(n+1)^2}$. Therefore
\begin{equation}\label{eq:Rn-asymptotic}
R_n
\asymp
\ell^{-2}2^{n+(n+1)^2}.
\end{equation}
In particular, $R_n\to\infty$. For every $\alpha>0$,
\[
R_n^\alpha L_n
\asymp
\ell^{1-2\alpha}
2^{\alpha(n+(n+1)^2)-n}
\longrightarrow
\infty.
\]
Set
\[
\xi_n=R_n\omega_n.
\]
Since $|\omega_n|=1$, one has $|\xi_n|=R_n\to\infty$. Equation
\eqref{eq:fourier-lower-final} therefore gives
\[
|\xi_n|^\alpha|\widehat\sigma_I(\xi_n)|
\geq
cR_n^\alpha L_n
\longrightarrow
\infty.
\]
This proves \eqref{eq:no-polynomial-decay}.
\end{proof}

\begin{remark}\label{rem:varying-normal-directions}
The frequencies used in the proof are not confined to one fixed ray. The
direction $\omega_n$ is the normal corresponding to the midpoint of the
slope range on $J_n$, and these directions may vary with $n$. This is
important: the theorem asserts failure of a uniform pointwise polynomial
bound in $\xi$, not failure of decay along one prescribed direction. The
almost-flat pieces supply a sequence of increasingly high frequencies at
which a definite portion of arclength is nearly stationary, while monotonicity
of the slope prevents the rest of the subarc from canceling that
contribution.
\end{remark}

\begin{corollary}\label{cor:strict-separation}
There is a strictly convex curve $\Gamma$ such that
\[
T(\Gamma)=1,
\]
while no nontrivial subarc of $\Gamma$ supports arclength with pointwise
Fourier decay of any positive power.
\end{corollary}

\begin{proof}
Take $\Gamma=\Gamma_1$ and combine Theorems
\ref{thm:gamma1-geometric-properties} and
\ref{thm:gamma1-no-polynomial-decay}.
\end{proof}

\section{A rectifiable endpoint theorem}\label{sec:rectifiable}

The preceding results impose geometric assumptions on the curve and then
allow an arbitrary compact set $E$ of dimension greater than one. At the
endpoint, a positive-measure conclusion also follows from rectifiability of
$E$ under a minimal assumption on the curve. Here a compact rectifiable curve
means the image of a continuous rectifiable path on a compact interval.

\begin{theorem}\label{thm:rectifiable-endpoint}
Let $\Gamma\subset\mathbb R^2$ be a compact rectifiable curve which is not
contained in a line. Let $E\subset\mathbb R^2$ contain a countably
$1$-rectifiable set $R$ with
\[
\mathcal H^1(R)>0.
\]
Then there are compact sets
\[
K_+\subset E+\Gamma
\qquad\text{and}\qquad
K_-\subset E-\Gamma
\]
such that $|K_+|>0$ and $|K_-|>0$.
\end{theorem}

\begin{proof}
By the definition of countable $1$-rectifiability, one of the Lipschitz
parametrizations covering $R$ up to an $\mathcal H^1$-null set has
positive-length intersection with $R$. Extending its two coordinate functions
to the line and restricting to a sufficiently large compact interval, we
obtain a compact interval $J\subset\mathbb R$ and a Lipschitz map
\[
e:J\longrightarrow\mathbb R^2
\]
such that
\[
\mathcal H^1(R\cap e(J))>0;
\]
see \cite[Chapter 15]{Mattila}. The set $e(J)$ has finite length. By inner
regularity, choose a compact set
\[
R_0\subset R\cap e(J)
\]
with $\mathcal H^1(R_0)>0$, and put
\[
A_0=e^{-1}(R_0).
\]
Then $A_0$ is compact and $e(A_0)=R_0$. For a map $f$ and a set $B$ in its
domain, write
\[
N(f,B,z)
=
\#\{w\in B:f(w)=z\}
\]
for the multiplicity of $f$ over $z$ on $B$. The one-dimensional area formula
\cite[Chapter 7]{Mattila} gives
\begin{align}
\int_{A_0}|e'(s)|\,ds
&=
\int_{\mathbb R^2}N(e,A_0,z)\,d\mathcal H^1(z)
\notag\\
&\geq
\mathcal H^1(R_0)
>
0.
\label{eq:positive-parameter-derivative}
\end{align}
Consequently, the measurable set
\[
A_1
=
\{s\in A_0:e'(s)\text{ exists and }|e'(s)|>0\}
\]
has positive Lebesgue measure. By inner regularity, $A_1$ contains a compact
subset $A$ of positive measure. Thus $e'(s)$ exists and is nonzero for every
$s\in A$, and
\[
e(A)\subset R_0\subset R\subset E.
\]

Rectifiability of $\Gamma$ gives an arclength parametrization
\[
r:[0,L]\longrightarrow\mathbb R^2
\]
whose image is $\Gamma$ and for which $|r'(t)|=1$ almost everywhere. The
assumption that $\Gamma$ is not contained in a line implies, in particular,
that $L>0$. Define
\[
F_+(s,t)
=
e(s)+r(t),
\qquad
(s,t)\in J\times[0,L].
\]
This map is Lipschitz. Wherever both $e'(s)$ and $r'(t)$ exist, its
differential has columns $e'(s)$ and $r'(t)$. By Rademacher's theorem and
Fubini's theorem, this holds for almost every $(s,t)$, and hence the
approximate Jacobian on $A\times[0,L]$ is
\begin{equation}\label{eq:rectifiable-jacobian}
J_{F_+}(s,t)
=
|\det(e'(s),r'(t))|
\end{equation}
for almost every $(s,t)\in A\times[0,L]$.

Fix $s\in A$. If the determinant in \eqref{eq:rectifiable-jacobian} vanished
for almost every $t$, let $n_s$ be a unit vector perpendicular to $e'(s)$.
Up to a choice of orientation,
\[
\frac{d}{dt}\bigl(r(t)\cdot n_s\bigr)
=
r'(t)\cdot n_s
=
\frac{\det(e'(s),r'(t))}{|e'(s)|}
\]
for almost every $t$. The function $t\mapsto r(t)\cdot n_s$ is absolutely
continuous, so it would be constant. This would place the image of $r$, and
hence $\Gamma$, in an affine line parallel to $e'(s)$, contrary to the
hypothesis. Therefore
\[
\int_0^L|\det(e'(s),r'(t))|\,dt>0
\]
for every $s\in A$. The inner integral is a nonnegative measurable function
of $s$. Since $|A|>0$, Tonelli's theorem gives
\begin{equation}\label{eq:positive-total-jacobian}
\iint_{A\times[0,L]}J_{F_+}(s,t)\,ds\,dt
>
0.
\end{equation}

The two-dimensional area formula gives
\[
\iint_{A\times[0,L]}J_{F_+}(s,t)\,ds\,dt
=
\int_{\mathbb R^2}N(F_+,A\times[0,L],z)\,dz.
\]
If $F_+(A\times[0,L])$ had planar measure zero, the right side would vanish,
contradicting \eqref{eq:positive-total-jacobian}. Hence
\[
K_+
=
F_+(A\times[0,L])
\]
has positive planar measure. It is compact because $A\times[0,L]$ is compact,
and it is contained in $E+\Gamma$ because $e(A)\subset E$.

For the difference set, define
\[
F_-(s,t)=e(s)-r(t).
\]
Its approximate Jacobian is
\[
|\det(e'(s),-r'(t))|
=
|\det(e'(s),r'(t))|.
\]
The same argument shows that
\[
K_-
=
F_-(A\times[0,L])
\]
is a compact set of positive measure contained in $E-\Gamma$.
\end{proof}

\begin{remark}
Theorem \ref{thm:rectifiable-endpoint} does not assert a universal result for
all sets of Hausdorff dimension one. For a smooth curve of nonvanishing
curvature, Simon and Taylor proved that a finite-length $1$-set has a null sum
precisely when it is purely unrectifiable \cite{SimonTaylor}. The theorem
above is the positive half of that picture under the weaker assumption that
$\Gamma$ is merely rectifiable and not a line.
\end{remark}

\section{The Fourier benchmark and remaining questions}
\label{sec:fourier-comparison}

We finish by recording the standard Fourier implication in its shortest
form. This makes precise what the examples in Sections
\ref{sec:nonrajchman} and \ref{sec:strict-example} go beyond.

\begin{proposition}[The direct Fourier criterion]\label{prop:direct-fourier}
Let $\Gamma\subset\mathbb R^2$ be compact, and suppose that there is a
nonzero finite positive measure $\sigma$ supported on $\Gamma$ such that, for
some $0<\alpha<1$,
\begin{equation}\label{eq:direct-decay}
|\widehat\sigma(\xi)|
\leq
C(1+|\xi|)^{-\alpha}.
\end{equation}
If $E\subset\mathbb R^2$ is compact and
\begin{equation}\label{eq:direct-fourier-threshold}
\dim_{\mathrm H}(E)>2-2\alpha,
\end{equation}
then
\[
|E+\Gamma|>0.
\]
\end{proposition}

\begin{proof}
Put $s=2-2\alpha$. Recall that
\[
I_s(\mu)
=
\iint |x-y|^{-s}\,d\mu(x)\,d\mu(y)
\]
for a finite positive measure $\mu$. Choose
\[
s<\beta<\dim_{\mathrm H}(E).
\]
By Frostman's lemma \cite{Mattila}, there is a nonzero finite measure $\mu$
supported on $E$ such that
\[
\mu(B(x,r))\leq C_\mu r^\beta.
\]
The standard passage from a $\beta$-Frostman estimate to finite $s$-energy is
an annular summation. Indeed,
\begin{align*}
I_s(\mu)
&\leq
\mu(E)^2
+
\sum_{j=0}^\infty
2^{(j+1)s}
\int_E\mu(B(x,2^{-j}))\,d\mu(x)
\\
&\leq
\mu(E)^2
+
C\mu(E)\sum_{j=0}^\infty2^{-j(\beta-s)}
<
\infty.
\end{align*}
The Riesz-energy identity \cite{Mattila} gives
\[
I_s(\mu)
=
c_s\int_{\mathbb R^2}
|\widehat\mu(\xi)|^2|\xi|^{s-2}\,d\xi.
\]
Since $s-2=-2\alpha$, this identity and \eqref{eq:direct-decay} imply
\begin{align*}
\int_{\mathbb R^2}
|\widehat\mu(\xi)\widehat\sigma(\xi)|^2\,d\xi
&\leq
C\int_{\mathbb R^2}
|\widehat\mu(\xi)|^2|\xi|^{-2\alpha}\,d\xi
\\
&\leq
C I_s(\mu).
\end{align*}
By Plancherel's theorem, $\mu*\sigma$ has a nonnegative $L^2$ density $f$.
The measure $\mu*\sigma$ is supported on $E+\Gamma$ and has total mass
$\mu(E)\sigma(\Gamma)>0$. Cauchy--Schwarz therefore gives
\begin{align}
\bigl(\mu(E)\sigma(\Gamma)\bigr)^2
&=
\left(\int_{E+\Gamma}f(z)\,dz\right)^2
\notag\\
&\leq
|E+\Gamma|\,\|f\|_2^2.
\label{eq:direct-fourier-quantitative}
\end{align}
Consequently,
\[
|E+\Gamma|
\geq
\frac{\bigl(\mu(E)\sigma(\Gamma)\bigr)^2}
{\|\mu*\sigma\|_2^2}
>
0.
\]
\end{proof}

For a compact rectifiable curve $\Gamma$, define its arcwise polynomial
decay exponent by
\begin{equation}\label{eq:arcwise-decay-exponent}
\alpha_{\mathrm{arc}}(\Gamma)
=
\sup
\left\{
\begin{aligned}
\alpha>0:
&\text{ there is a nontrivial subarc }\Gamma_0\subset\Gamma\\
&\text{ such that }
|\widehat{\mathcal H^1|_{\Gamma_0}}(\xi)|
\lesssim
(1+|\xi|)^{-\alpha}
\end{aligned}
\right\}.
\end{equation}
If the set in \eqref{eq:arcwise-decay-exponent} is empty, we set
$\alpha_{\mathrm{arc}}(\Gamma)=0$.

The quantity $\alpha_{\mathrm{arc}}(\Gamma)$ is meant only as a convenient
way to compare the geometric threshold with the classical arclength
argument. For a $C^2$ arc with curvature bounded away from zero, stationary
phase gives the exponent $\frac12$, and Proposition
\ref{prop:direct-fourier} then recovers the threshold
$\dim_{\mathrm H}(E)>1$. The strictly convex example constructed above has the same
geometric threshold even though this local arclength exponent drops all the
way to zero. In that sense the example separates two statements which agree
for the usual smooth models: the existence of pointwise Fourier decay for
arclength and the positive-measure behavior of all translates indexed by a
set of dimension greater than one.

It is important that the definition is arcwise. If one used only arclength
on the whole curve, a single poorly behaved portion could destroy a global
Fourier estimate while another curved subarc retained enough decay to prove
$T(\Gamma)=1$ by the direct criterion. Requiring the obstruction on every
nontrivial subarc rules out this possibility. On the other hand,
$\alpha_{\mathrm{arc}}(\Gamma)=0$ does not say that every measure supported
on every subarc has Fourier dimension zero, nor is such a statement needed
for our purpose. The comparison here is with the natural arclength measure
that appears in the standard convolution argument.

\begin{corollary}\label{cor:T-one-alpha-zero}
There is a strictly convex Lipschitz graph $\Gamma$ such that
\[
T(\Gamma)=1
\qquad\text{and}\qquad
\alpha_{\mathrm{arc}}(\Gamma)=0.
\]
There is also a convex Lipschitz graph with $T(\Gamma)=1$ for which arclength
on every nontrivial subarc fails to be Rajchman.
\end{corollary}

\begin{proof}
The first assertion is Corollary \ref{cor:strict-separation}. The second is
Theorem \ref{thm:gamma0-nonrajchman}.
\end{proof}

Thus polynomial pointwise decay of arclength on a positive-length subarc is
not necessary for the optimal universal Minkowski-sum threshold. The
mechanism is geometric: a positive-length curved trace supplies enough
translated-tube transversality even when the arclength measure on every
subarc has severe Fourier obstructions. We do not claim that the curve
supports no other specially chosen measure having some Fourier decay; the
separation established here concerns the natural arclength measures that
enter the usual curve-averaging argument.

Theorems \ref{thm:positive-curved-trace} and \ref{thm:ac-curvature} do not
treat convex graphs whose curvature measure is singular. We first record a
converse observation which explains why the curved-trace mechanism cannot be
used in that regime.

\begin{proposition}[Purely singular curvature]
\label{prop:singular-no-trace}
Let $\gamma:[a,b]\to\mathbb R$ be convex and Lipschitz, and let $\Gamma$ be
its graph. If $D^2\gamma$ is singular with respect to Lebesgue measure, then
$\Gamma$ has no positive curved trace.
\end{proposition}

\begin{proof}
Suppose instead that $\Gamma$ has a positive curved trace on
\[
\Sigma=\{(x,g(x)):x\in I\},
\]
where $g\in C^2(I)$ and $|g''|\geq c_0>0$. Let $A\subset\Gamma\cap\Sigma$
be the trace and let $S$ be its projection onto the first coordinate. Since
the graph parametrization of $\Sigma$ is bi-Lipschitz, $|S|>0$, and
$\gamma=g$ on $S$.

Choose a point $x_0\in S$ which is a Lebesgue density point of $S$ and at
which Lemma \ref{lem:convex-second-order} applies to $\gamma$. Such points
form a full-measure subset of $S$. Choose $x_j\in S\setminus\{x_0\}$ with
$x_j\to x_0$ and put $u_j=x_j-x_0$. Taylor's formula for $g$ and
\eqref{eq:convex-second-order} give
\begin{align*}
0
&=
(\gamma-g)(x_0+u_j)-(\gamma-g)(x_0)
\\
&=
(\gamma'(x_0)-g'(x_0))u_j
+
\frac12(k(x_0)-g''(x_0))u_j^2
+
o(u_j^2).
\end{align*}
Dividing first by $u_j$ and then, after the linear coefficient is known to
vanish, by $u_j^2$, gives
\[
\gamma'(x_0)=g'(x_0),
\qquad
k(x_0)=g''(x_0).
\]
Since $k\geq0$ almost everywhere and $|g''|\geq c_0$, it follows that
$k(x_0)\geq c_0$ at almost every such density point of $S$. Hence $k$ is
positive on a set of positive measure, contradicting the singularity of
$D^2\gamma$. This proves the proposition.
\end{proof}

A quantitative lower growth condition on the curvature nevertheless gives an
intermediate threshold. The following result also indicates what strict
convexity by itself does not provide.

\begin{proposition}[Quantitative lower growth of curvature]
\label{prop:curvature-lower-growth}
Let $\gamma:[a,b]\to\mathbb R$ be convex and Lipschitz, let $\Gamma$ be its
graph, and let $p=\gamma'_+$ be its right derivative. Suppose that for some
$1\leq\theta<2$ and $c_0>0$,
\begin{equation}\label{eq:curvature-lower-growth}
p(t)-p(s)
\geq
c_0(t-s)^\theta
\end{equation}
whenever $a\leq s<t<b$. Then, for every compact set
$E\subset\mathbb R^2$,
\[
\dim_{\mathrm H}(E)>\theta
\quad\Longrightarrow\quad
|E+\Gamma|>0.
\]
Consequently,
\[
T(\Gamma)\leq\theta.
\]
\end{proposition}

\begin{proof}
Write $L=\|p\|_\infty$ and fix $0<\delta<\delta_0$, where
$0<\delta_0\leq1$. As in the proof of Theorem
\ref{thm:curved-graph-overlap}, the Euclidean $\delta$-neighborhood of
$\Gamma$ is contained in the vertical neighborhood
\[
V_\delta
=
\{(x,y):a\leq x\leq b,\ |y-\gamma(x)|\leq C_0\delta\}
\]
together with two endpoint disks of radius $C_0\delta$.

Let $h=(u,v)$. If $u>0$, put
\[
G_u(x)=\gamma(x)-\gamma(x-u)
\]
on the interval where both terms are defined. At every $x$ for which $\gamma'$ exists at both $x$ and $x-u$, a
full-measure set in the common domain,
\[
G_u'(x)
=
p(x)-p(x-u)
\geq
c_0u^\theta
\]
by \eqref{eq:curvature-lower-growth}. Since $G_u$ is absolutely continuous,
the set of $x$ satisfying
\[
|G_u(x)-v|\leq2C_0\delta
\]
has length at most $C\delta/u^\theta$. The same conclusion holds for $u<0$,
with $|u|$ in place of $u$, because the derivative then has the opposite sign
and the same lower bound in absolute value. Since every vertical section of
the overlap has length at most $2C_0\delta$,
\begin{equation}\label{eq:theta-vertical-overlap}
|V_\delta\cap(h+V_\delta)|
\leq
C\delta
\min\left\{1,\frac{\delta}{|u|^\theta}\right\},
\end{equation}
where the minimum is interpreted as $1$ when $u=0$.

If the overlap is nonempty, the Lipschitz property gives
\[
|h|
\leq
(1+L)|u|+2C_0\delta.
\]
When $|h|\leq C_1\delta^{1/\theta}$, one has
$\delta+|h|^\theta\leq C\delta$, and the trivial bound
$|V_\delta|\leq C\delta$ yields
\[
|V_\delta\cap(h+V_\delta)|
\leq
C\frac{\delta^2}{\delta+|h|^\theta}.
\]
When $|h|>C_1\delta^{1/\theta}$, note that
$\delta\leq\delta^{1/\theta}$ because $\delta\leq1$ and $\theta\geq1$.
Choosing $C_1>4C_0$, the preceding comparison gives
\[
|u|
\geq
\frac{|h|}{2(1+L)}.
\]
In this regime $|u|^\theta\gtrsim|h|^\theta\gtrsim\delta$, so the second
term in the minimum in \eqref{eq:theta-vertical-overlap} applies and gives
\[
|V_\delta\cap(h+V_\delta)|
\leq
C\frac{\delta^2}{|h|^\theta}
\leq
C\frac{\delta^2}{\delta+|h|^\theta}.
\]

Every intersection involving an endpoint disk has area $O(\delta^2)$. If
such an intersection is nonempty, then $|h|$ is bounded in terms of
$\Gamma$. Consequently, $\delta+|h|^\theta\lesssim_\Gamma 1$, and its area
is bounded by $C\delta^2/(\delta+|h|^\theta)$. Summing the finitely many
intersections, we have proved
\begin{equation}\label{eq:theta-graph-overlap}
|\Gamma^\delta\cap(h+\Gamma^\delta)|
\leq
C\frac{\delta^2}{\delta+|h|^\theta}.
\end{equation}

Lemma \ref{lem:trace-tube-volume}, applied with $S=[a,b]$, gives
$|\Gamma^\delta|\geq c\delta$. Let $\mu$ be a nonzero finite positive measure
supported on $E$ with
\[
I_\theta(\mu)
=
\iint|x-y|^{-\theta}\,d\mu(x)\,d\mu(y)
<
\infty.
\]
Define
\[
F_\delta(z)
=
\int_E\frac{1_{x+\Gamma^\delta}(z)}{|\Gamma^\delta|}\,d\mu(x).
\]
Then $F_\delta$ is supported on $E+\Gamma^\delta$, has integral $\mu(E)$,
and, by \eqref{eq:theta-graph-overlap},
\begin{align*}
\|F_\delta\|_2^2
&=
\iint
\frac{|\Gamma^\delta\cap((y-x)+\Gamma^\delta)|}
{|\Gamma^\delta|^2}
\,d\mu(x)\,d\mu(y)
\\
&\leq
C\iint\frac{d\mu(x)\,d\mu(y)}
{\delta+|x-y|^\theta}
\\
&\leq
C I_\theta(\mu).
\end{align*}
Cauchy--Schwarz therefore gives
\[
|E+\Gamma^\delta|
\geq
c\frac{\mu(E)^2}{I_\theta(\mu)},
\]
uniformly in $\delta$. For a decreasing sequence $\delta_j\downarrow0$, the
compact sets $E+\Gamma^{\delta_j}$ decrease to $E+\Gamma$ by
\eqref{eq:intersection-thickened-sums}. Continuity from above of Lebesgue
measure gives
\[
|E+\Gamma|
\geq
c\frac{\mu(E)^2}{I_\theta(\mu)}
>
0.
\]

Finally, if $\dim_{\mathrm H}(E)>\theta$, choose
\[
\theta<\beta<\dim_{\mathrm H}(E).
\]
Frostman's lemma gives a nonzero $\beta$-Frostman measure on $E$, and the
same annular summation used in the proof of Proposition
\ref{prop:direct-fourier} gives $I_\theta(\mu)<\infty$. The preceding argument
applies, proving the proposition.
\end{proof}

Condition \eqref{eq:curvature-lower-growth} is equivalent, up to the harmless
endpoint convention for the monotone slope, to the lower bound
\[
D^2\gamma((s,t])\geq c_0(t-s)^\theta.
\]
For $\theta=1$ it forces a nonzero absolutely continuous curvature component.
For $\theta>1$ the proposition gives the partial bound
$T(\Gamma)\leq\theta$ whenever such lower growth is available in a singular
regime. Strict convexity alone says only that
$p(t)>p(s)$ whenever $s<t$; it supplies no uniform power law of this kind.
This leads to the following sharper question.

\begin{question}\label{question:strict-convexity}
Does every strictly convex Lipschitz graph $\Gamma\subset\mathbb R^2$ satisfy
\[
T(\Gamma)=1?
\]
\end{question}

The question includes both dense atomic curvature and purely singular
continuous curvature. A satisfactory answer should distinguish geometric
transversality from the pointwise decay of one fixed arclength measure. The
examples in this paper show that the latter cannot be a necessary condition.

\end{document}